\documentclass[8pt, a4paper]{amsart}

\usepackage{amssymb, array, amsmath, amscd, pdfpages, enumerate, amsthm, setspace, hyperref,mathtools, natbib}

\usepackage[papersize={15cm,25cm},total={12.7cm,22.5cm},top=1.5cm,left=1.2cm]{geometry} 

\usepackage[utf8]{inputenc}
\usepackage{amssymb}
\usepackage{bm}
\usepackage{graphicx} 

\usepackage{overpic}

\newcommand{\C}{\mathbb{C}}						% Complex numbers
\newcommand{\Z}{\mathbb{Z}}						% Integers
\newcommand{\R}{\mathbb{R}}						% Real numbers
\renewcommand{\i}{\mathrm{i}}					% Imaginary unit
\newcommand{\e}{\mathrm{e}}						% Euler number
\newcommand{\inv}{^{-1}}
\newcommand{\List}[1]{\{1,\ldots, #1\}}

\newcommand{\matrices}[1]{\mathcal{M}\left( #1 \right)}		% The set of matrices over #1
\newcommand{\diag}[1]{\mathrm{diag}\left( #1 \right)}		% Diagonal matrix with elements #1
\newcommand{\RHinf}{RH^\infty}								% Stable rational functions
\newcommand{\Hinfty}{H^\infty}								% H infty functions
\newcommand{\pstable}{\mathbf{P}}							% Domain of P-stable functions
\newcommand{\error}{e}										% An error signal
\newcommand{\control}{u}									% A control signal
\newcommand{\measurement}{y}								% A measurement signal
\newcommand{\reference}{\measurement_r}						% A reference signal
\newcommand{\disturbance}{d}								% A disturbance signal

\newcommand{\Plant}{P}										% Plant transfer function
\newcommand{\Cont}{C}										% Controller
\newcommand{\Gen}{\Theta}									% Signal generator
\newcommand{\gen}{\theta}									% Element of the signal generator
\newcommand{\Closedloop}[1]{H( #1 )}						% Closed loop transfer function

\newcommand{\num}{N}										% Numerator
\newcommand{\den}{D}										% Denominator
\newcommand{\lnum}{\widetilde{\num}}						% Numerator of left facorization
\newcommand{\lden}{\widetilde{\den}}						% Denominator of left factorization
\newcommand{\outd}{n}										% Dimension of the output space
\newcommand{\ind}{m}										% Dimension of the input space
\newcommand{\stable}{\mathbf{R}}							% Domain of stable functions (arbitrary)
\newcommand{\primeideal}{\mathbf{Z}}						% Prime ideal for denomiators of causal transfer functions
\newcommand{\fractions}[1]{\mathbf{F}_{#1}}					% Field of fractions of stable transfer functions
\newtheorem{theorem}{Theorem}
\newtheorem{lemma}{Lemma}
\newtheorem{corollary}{Corollary}
\theoremstyle{definition}
\newtheorem{definition}{Definition}

\newtheorem{remark}{Remark}
\newtheorem{example}{Example}
\newtheorem{design}{Design procedure}

\begin{document}

\title[A fractional representation approach to robust regulation]{A fractional representation approach to the robust regulation problem for MIMO systems}

\thispagestyle{plain}

\author{Petteri Laakkkonen}
\address{Unit of Computing Sciences, Faculty of Information Technology and Communication Sciences, Tampere University, PO. Box 692, 33014, Tampere, Finland.}
\email{petteri.laakkonen@tuni.fi}

\begin{abstract}
The aim of this paper is in developing unifying frequency domain theory for robust regulation of MIMO systems. The main theoretical results achieved are a new formulation of the internal model principle, solvability conditions for the robust regulation problem, and a parametrization of all robustly regulating controllers. The main results are formulated with minimal assumptions and without using coprime factorizations, thus guaranteeing their applicability with a very general class of systems. In addition to theoretical results, the design of robust controllers is addressed. The results are illustrated by two examples involving a delay and a heat equation.
\end{abstract}


\maketitle

\setcitestyle{numbers,empty,super}

%%%%%%%%%%%%%%%%%%%%%%%%%%%%%%%%%%%%%%%%%%%%%%%%%%%%%%%%%%%%%
%% ==> Introduction and problem formulation (and organization of the article):
\section{Introduction}

% DESCRIPTION OF THE PROBLEM CONSIDERED

Controlling behavior of infinite-dimensional systems, e.g., systems described by partial differential equations or time-delay systems, is of great interest in many applications. This paper studies the frequency domain formulation of the control problem where a dynamic controller is to be found so that the output $y(t)$ of the system asymptotically converges to the given reference signal $\reference(t)$, i.e., $\|y(t)-\reference(t)\|\to 0$ as $t\to\infty$.
%, despite some external disturbance $d(t)$.
%The problem studied in this paper is as follows. One is given a plant $P$. Information about the behavior of the plant is achieved through a measured output $y$ and the behavior can be affected by a control input $u$. The aim of the control problem is to find a dynamical error feedback controller $C$ making the system to behave in a desired manner. 
%Here this means that the output asymptotically converges to the given reference signal $y_r(t)$ with respect to time despite some external disturbance $d(t)$. 
The controllers achieving asymptotic convergence are said to be regulating. In addition, the controller is required to work despite small perturbations of the plant. 
%Such perturbation have multiple sources, e.g., modeling and estimation errors
% or finite-dimensional approximations in the controller design process \cite{Williametal2002}. 
This property is called robustness and it is important in real world control applications since the related system models, the controller design procedures, testing feasibility of the controller by simulations, and finally implementing the controllers in practice unavoidably involve some inaccuracies.\cite{Williametal2002} 
%caused by approximations and estimation errors 
%\cite{Williametal2002}. 
The problem of finding a regulating controller that is robust to small perturbations is called the robust regulation problem. 
As explained by Paunonen and Pohjolainen\cite{PaunonenPohjolainen2010}, the robust regulation problem can be divided into two equally important parts: the robust stabilization and the robust regulation. In the stabilization part one is interested in finding a controller such that stability of the closed loop is maintained despite small perturbations of the plant. It involves the topological aspects of the problem, and it has been studied in the algebraic setting\cite{BallSasane2012,Quadrat2015,Vidyasagaretal1982} as well as in many physically relevant algebraic structures\cite{CurtainZwart1995, Vidyasagar}. This paper focus on the regulation part where one is interested in finding conditions under which stability implies regulation. This characterization of robust regulation has been used for example by Paunonen and Pohjolainen\cite{PaunonenPohjolainen2010} in the time domain and by Callier and Desoer\cite{CallierDesoer1980} in the frequency domain.

%The robust regulation problem, i.e., the problem where the output of the plant should converge asymptotically to a given reference signal with respect to time despite small perturbations in the plant, is studied in this paper in the frequency domain. Controllers achieving asymptotic convergence are called regulating and the tolerance to errors or small changes is called robustness. Robustness is of great importance since when designing, constructing, and applying controllers for real world applications one has several sources of errors in the process, e.g., the mathematical model cannot exactly capture the phenomenon described or the design process and testing of the real world controller rely on finite-dimensional approximations \cite{Williametal2002}.

% EARLIER RESEARCH / LITERATURE REVIEW

Regulation of systems modelled with ordinary differential equations achieved considerable attention in the 1970's. \cite{AntsaklisPearson1978, Davison1976a, DavisonGoldenberg1975, FrancisWonham1976, FrancisWonham1975a} The results have been generalized to infinite-dimensional systems since then. \cite{Pohjolainen1985, YamamotoHara1988, Byrnesetal2000, HamalainenPohjolainen2000, RebarberWeiss2003, ImmonenPhd, PaunonenPohjolainen2014, LaakkonenPohjolainen2015, Humalojaetal2019} A remarkably important result called the internal model principle was given by Francis and Wonham\cite{FrancisWonham1975a} and by Davison\cite{Davison1976a}. It states that all robustly regulating controllers must contain an internal model, i.e., a suitably reduplicated copy, of the unstable dynamics of the reference signals. The internal model principle has multiple different time domain characterizations.\cite{ImmonenPhd, PaunonenPohjolainen2014} A frequency domain formulation of the internal model principle and solvability conditions for the robust regulation problem were given by Vidyasagar\cite{Vidyasagar} for rational transfer functions. These results were later generalized to specific classes of transfer functions suitable for infinite-dimensional systems\cite{YamamotoHara1988, LaakkonenPohjolainen2015} and by using the fractional representation approach \cite{NettThesis, LaakkonenQuadrat2017}. Frequency domain methods for designing regulating controllers have also been considered by several authors.\cite{HamalainenPohjolainen2000, RebarberWeiss2003, LaakkonenPohjolainen2015, Ylinenetal2006}

% ALGEBRAIC FRAMEWORK: JUSTIFICATION FOR CHOICE
% ADVANTAGES OF FRACTIONAL REPRESENTATIONS

In this paper, the robust regulation problem is studied using the fractional representation approach. \cite{Vidyasagar,Desoeretal1980} Fractional representations have two benefits. First, fractional representations allow considerations to be done only assuming that stable SISO %single-input single-output 
transfer functions form a commutative ring $\stable$ with no zero divisors. Within this general algebraic framework regulation simply means that the error $e=y-\reference$ between the output of the system and the reference is a stable vector, i.e., a vector with elements in the ring $\stable$. 
The posed natural algebraic conditions are valid for most classes of transfer functions.\cite{LaakkonenPohjolainen2015, Logemann1993,  PekarProkop2017} Consequently, the results of this paper provide a simple framework to study robust regulation that is applicable in a wide range of systems including, e.g., finite-dimensional, distributed parameter, and time-delay systems. This is crucial since the suitable choice of the ring $\stable$ depend on the properties of the system and the unstable dynamics to be regulated.
% Section \ref{sec:Preliminaries}. 
Secondly, fractional representations allow parametrizing all stabilizing controllers using no coprime factorizations.\cite{Quadrat2006} This has an instrumental role in the proofs of the main results in this paper and allow formulating them without coprime factorizations. This is an advantage since the coprime factorizations of infinite-dimensional systems are not easy to construct in general,
%\cite{EhkaBonnetjaAlban},
they may not exist\cite{PekarProkop2017, Logemann1987}, or the existence may be unknown\cite{Laakkonen2016}.

% THE MAIN RESULTS OF THE PAPER

The main achievement of this paper is in generalizing several classical frequency domain results on robust regulation of systems described by rational transfer matrices presented by Vidyasagar\cite{Vidyasagar}. The main theoretical results are a new formulation of the internal model principle, conditions for the existence of a robustly regulating controller, and a parametrization of all robust controllers. Causality is considered by using the stability results by Mori and Abe\cite{MoriAbe2001}. The main results extend the existing ones that are specific to some ring $\stable$ and use the coprime factorizations\cite{Vidyasagar, YamamotoHara1988, LaakkonenPohjolainen2015, NettThesis} to the abstract algebraic setting using no coprime factorizations, thus providing a unifying theory for robust regulation in the frequency domain. The theoretical results of this paper extend the frequency domain results of SISO %single-input single-output 
systems presented by Laakkonen and Quadrat\cite{LaakkonenQuadrat2017} to MIMO %multi-input multi-output 
systems. Unlike with the SISO systems, regulation does not imply robustness with the MIMO systems making the generalization of the results nontrivial. The results of this paper show how the $p$-copy internal model for time domain system\cite{PaunonenPohjolainen2010, FrancisWonham1975a} can be understood within the general framework adopted in this paper.
MIMO %Multi-input multi-output 
formulation of the internal model principle in the general algebraic framework was first considered in the conference paper by the author. \cite{Laakkonen2017} The solvability of the robust regulation problem, parametrization of robust controllers, or the controller design were not addressed. This paper extends the preliminary results of Laakkonen\cite{Laakkonen2017} by introducing a new reformulation of the internal model principle and discussing the results and presenting their proofs in greater detail. In particular, the sufficiency part of the internal model principle now address also robustness whereas the original proof only stated that the internal model implies regulation. The preliminary results\cite{LaakkonenQuadrat2017, Laakkonen2017} do not take causality into account, so the results addressing it are all new.

In addition to theoretical results, several controller design procedures related to the given existence results are proposed. They generalize the ideas for constructing robustly regulating controllers \cite{Vidyasagar, LaakkonenPohjolainen2015, Laakkonen2016} to the general algebraic framework. In addition, a new method of constructing the internal model one element at a time is proposed. It allows revising an already existing controller by including additional parts into its internal model thus extending the class of regulated signals.

Two examples are given to illustrate the proposed controller design procedures and theoretical results. The first example involves a reference signal with an infinite number of unstable poles making the design procedure complicated. This demonstrates how the choice of the ring $\stable$ depends on the problem at hand and underlines the importance of the general approach. In the second example with one dimensional heat equation, 
%one first applies time domain methods to stabilize the plant and then frequency domain methods to include the internal model. This way one does 
it is shown that one may be able to carry out the design procedure using approximations of the plant transfer matrix. % instead of the precise values. 
This way one does not need to find the closed form of the plant transfer matrix which is a considerable benefit. In addition, the controller of this example is easily verified to be causal even though the controller design method does not directly imply causality. This shows that the results presented in this paper that do not directly imply causality are also relevant and interesting.

% STRUCTURE OF THE PAPER
The remaining part of this paper is structured as follows. The preliminary definitions, notations and stability results are introduced in Section \ref{sec:Preliminaries}. The problem formulation is given in Section \ref{sec:Problem}. Section \ref{sec:IMP} is devoted to the internal model principle.
% where it is formulated using fractional representations.
%The internal model principle is formulated using no coprime factorizations.% and it is shown that the new formulation is equivalent to the classical one provided a coprime factorization exists. 
In Section \ref{sec:Simplification}, simplification of the internal model is discussed in term of fractional ideals.
Solvability of the robust regulation problem is studied in Section \ref{sec:Solv}. 
%Several solvability conditions under different assumptions on the plant and reference signals are given and controller design procedures based on them are proposed. 
In addition, controller design is addressed and 
% whose unstable dynamics can be characterized by a simple stable transfer function,  
a parametrization of all robustly regulating controllers is proposed. The theoretical results and design procedures are illustrated by two examples in Section \ref{sec:Example}. Finally, the obtained results are summarized and discussed in Section \ref{sec:Conclusions}.

%% <== End of intro
%%%%%%%%%%%%%%%%%%%%%%%%%%%%%%%%%%%%%%%%%%%%%%%%%%%%%%%%%%%%%

\section{Notations and Preliminary Results}\label{sec:Preliminaries}

The set of stable causal SISO %single-input single-output 
transfer functions is denoted by  $\stable$ and together with the summation $+$ and multiplication $\cdot$ operations it is assumed to be an integral domain, i.e., a commutative ring with no zero divisors.
The following integral domains appear in the examples. 
The Hardy space of bounded holomorphic functions in the right half plane $\C_+=\{s\in\C\, |\, \mathrm{Re}(s)>0 \}$ is denoted by $\Hinfty$. The set of all real rational functions and its subset of proper rational functions having no poles in $\C_+$ are denoted by $\R(s)$ and $\RHinf$, respectively. The integral domain $\pstable$ consists of functions $f(s)$ that are analytic and bounded in every right-half plane $\C_\alpha=\{s\in \C\, |\, \mathrm{Re}(s)>\alpha \}$ with $\alpha>0$ and polynomially bounded on the imaginary axis, i.e., $|f(\i\omega)|<M|\omega|^k$ for some $M,k>0$.\cite{LaakkonenPohjolainen2015} This integral domain corresponds to polynomial stability in the time domain. \cite{PaunonenLaakkonen2015}

The additive and multiplicative identities of $\stable$ are denoted by $0$ and $1$, respectively. Invertible elements of $\stable$ are called units. A set $A\subseteq \stable$ is an ideal of $\stable$, if it is an additive subgroup of $\stable$ and $ab\in A$ whenever $a\in A$ and $b\in\stable$. An ideal $A$ is prime if $A\neq \stable$ and $a\in A$ or $b\in A$ if $ab\in A$. The field of fractions of $\stable$ is denoted by $\fractions{\stable}$. The $\stable$-module $f_1 \stable+\cdots+f_n\stable$, where $f_1,\ldots,f_k\in\fractions{\stable}$, is denoted by $\langle f_1,\ldots,f_k\rangle$ or $\langle f_i\, |\, i=1,\ldots, k\rangle$.

\begin{definition}\begin{enumerate}
\item An $\stable$-submodule $J$ of $\fractions{\stable}$ is called \emph{a fractional ideal} if there exists $0\neq a\in\stable$ such that $a J\subseteq\stable$.
\item A fractional ideal $J$ is \emph{finitely generated} if $J=\langle f_1,\ldots,f_k\rangle$ for some elements $f_1,\ldots,f_k\in\fractions{\stable}$ and it is
\emph{principal} if it is generated by a single element, i.e., $J=\langle f\rangle$ for some $f\in\fractions{\stable}$.
\end{enumerate}

%Given two fractional ideals $\I$ and $\J$, we define
%\begin{eqnarray*} 
%& IJ=\definedset{\sum_{i=1}^n a_i b_i}{a_i\in \I, b_i\in\J, n\in\N}
%\intertext{and}
%& I:J=\definedset{k\in \fractions{\stable}}{(k)\J\subseteq \I}.
%\end{eqnarray*}
%We say that a fractional ideal $\I$ is invertible is there exists a fractional ideal  $\J$ such that $\I\J=\stable$. The fractional ideal $\J$ is called the inverse of $\I$, and we denote it by $\I\inv $.
\end{definition}

A matrix $H$ with elements $\gen_{ij}$ on the $i$th row and $j$th column is denoted by $H=(\gen_{ij})$ and its transpose is denoted by $H^T$. The set of all matrices with elements in a set $S$ is denoted by $\matrices{S}$ and the set of all $n\times m$ matrices by $S^{n\times m}$. The set of $n$-dimensional vectors with elements in $S$ is denoted by $S^n$.

The plant and the controller are assumed to be matrices over the field of fractions $\fractions{\stable}$. The control configuration considered in this article is depicted in Fig. \ref{fig:Closedloop}. 
The resulting closed loop transfer matrix from $(\reference,\disturbance)$ to $(\error,\control)$ is
\begin{eqnarray}\label{eq:ClosedLoop}
\Closedloop{\Plant,\Cont}  :=
\begin{bmatrix}
\left(I-\Plant\Cont\right)\inv  & \left(I-\Plant\Cont\right)\inv \Plant\\
\Cont\left(I-\Plant\Cont\right)\inv  &\left(I-\Cont\Plant\right)\inv 
\end{bmatrix}\in\matrices{\fractions{\stable}}.
\end{eqnarray}
The transfer functions $(I-\Plant\Cont)\inv$ and $(I-\Plant\Cont)\inv \Plant$ are called the \emph{sensitivity matrix} and the \emph{load disturbance sensitivity matrix}, respectively.

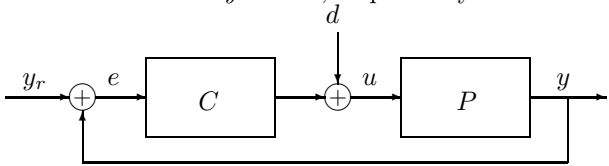
\begin{figure}[ht]
\centering
\begin{overpic}[scale=.9]{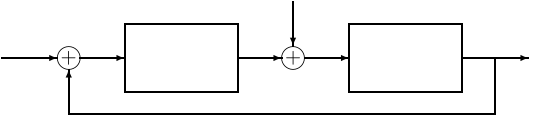}
\put(3,13){$\reference$}
\put(17,13){$\error$}
\put(32,9){$\Cont$}
\put(59,13){$\control$}
\put(53,23){$\disturbance$}
\put(74.5,9){$\Plant$}
\put(91,13){$\measurement$}
\end{overpic}
\caption{The error feedback control configuration.}
\label{fig:Closedloop}
\end{figure}

\begin{remark}
%It was mentioned in the introduction that the error is the difference between $y$ and $y_r$. This corresponds to the time domain intuition that regulation means convergence in of $y(t)$ to $\reference(t)$. However, here 
The choice $e=y+\reference$ is made in Fig. \ref{fig:Closedloop} instead of the more intuitive $e=y-y_r$. It does not restrict the generality, and one can avoid some technical difficulties this way because the closed loop is then symmetric with respect to $\Plant$ and $\Cont$.
\end{remark}

\begin{definition}
\begin{enumerate}\label{def:Stability}
\item A matrix or a vector $H\in\matrices{\fractions{\stable}}$ is \emph{stable} if $H\in\matrices{\stable}$, and otherwise it is unstable.

\item A controller $\Cont\in\fractions{\stable}^{m\times n}$ \emph{stabilizes} $\Plant\in\fractions{\stable}^{n\times m}$ if the closed loop transfer matrix \eqref{eq:ClosedLoop} is stable.
\end{enumerate}
\end{definition}

\begin{remark}\label{rem:IOStability}
If $\stable=\Hinfty$, then Definition \ref{def:Stability} of stability means that an input signal in the Lebesgue space $L^2(0,\infty,\C^\ind)$ produces $L^2(0,\infty,\C^\outd)$-output signal in the time-domain. Furthermore, under the standard assumptions that the system operator of the state-space representation generates a strongly stable semigroup, this means that the output converges to zero. \cite{Oostveen2000}
%\cite[Lemma 2.1.3]{Oostveen2000}.
\end{remark}

\begin{definition}
\begin{enumerate}
\item 
The representation $H=\num\den\inv $ ($H=\lden\inv \lnum$) is called \emph{a right (left) factorization} of $H$ if $\num,\den\in\matrices{\stable}$ ($\lnum,\lden \in\matrices{\stable}$) and $\det(\den)\neq 0$ ($\det(\lden)\neq 0$).
\item 
The factorization $H=\num\den\inv$ ($H=\lden\inv \lnum$) 
is 
%called \emph{a right (left) weakly coprime factorization}
\emph{right (left) weakly coprime} if for any $X\in\matrices{\fractions{\stable}}$ ($\widetilde{X}\in\matrices{\stable}$) of suitable size one has
\begin{equation*}
\begin{bmatrix}
N\\
D
\end{bmatrix}X\in\matrices{\stable}\Rightarrow X\in\matrices{\stable} \qquad \left(\widetilde{X}\begin{bmatrix}
\lnum &
\lden
\end{bmatrix}\in\matrices{\stable}\Rightarrow \widetilde{X}\in\matrices{\stable}\right).
\end{equation*}
\item 
The factorization $H=\num\den\inv $ ($H=\lden\inv \lnum$) 
is 
%called \emph{a right (left) coprime factorization}
\emph{right (left) coprime}
% of $\Gen$ 
if there exist $X,Y\in\matrices{\stable}$ ($\widetilde{X},\widetilde{Y}\in\matrices{\stable}$) such that
$
X\den-Y\num=I$ $(\lden\widetilde{X} - \lnum\widetilde{Y} =I).
$
\end{enumerate}
\end{definition}

Any right (left) coprime factorization is a weakly right (left) coprime factorization. It follows that the results assuming weakly coprime factorizations are valid if the factorization is coprime. In general, weakly coprime factorizations need not be coprime. However, a weakly right (left) coprime factorization of a stabilizing controller $\Cont$ or a stabilizable plant $\Plant$ is right (left) coprime.\cite{Quadrat2006b}
% \cite[Corollary 8]{Quadrat2006b}.

In what follows the stability results given in the next theorem are used extensively. The first item is Theorem 3 of Quadrat \cite{Quadrat2006} reformulated using Proposition 4 of the same article. It gives a parametrization of all stabilizing controllers. The second part is obtained from the first one by changing the roles of $\Plant$ and $\Cont$. It holds by the symmetry of the closed loop control configuration of Fig. \ref{fig:Closedloop}.
\begin{theorem}
\label{thm:Stability}
Let $\Cont\in\fractions{\stable}^{m\times n}$ 
stabilize $\Plant\in\fractions{\stable}^{n\times m}$.
\begin{enumerate}[1.]
\item Denote
\begin{eqnarray*}
\widetilde{L} :=\begin{bmatrix}
-\left(I-\Cont\Plant\right)\inv \Cont\;\; & \left(I-\Cont\Plant\right)\inv
\end{bmatrix}, %\in\fractions{\stable}^{m\times (n+m)},
%\end{eqnarray*}
\quad \text{and}\quad
%\begin{eqnarray*}
L :=\begin{bmatrix}
\left(I-\Plant\Cont\right)\inv \\
\Cont\left(I-\Plant\Cont\right)\inv
\end{bmatrix}. %\in\fractions{\stable}^{(n+m)\times n}.
\end{eqnarray*}
All stabilizing controllers of $\Plant$ are parametrized by
\begin{subequations}\label{eq:AllStabilizingControllers}
\begin{eqnarray}
\Cont(W) & =\left(\Cont\left(I-\Plant\Cont\right)\inv +\widetilde{L}WL\right)
%\notag\\
%& \qquad \times
\left(\left(I-\Plant\Cont\right)\inv +\Plant \widetilde{L}WL\right)\inv \\
& =\left(\left(I-\Cont\Plant\right)\inv +\widetilde{L}WL\Plant\right)\inv %\notag\\
%&\qquad \times
\left(\left(I-\Cont\Plant\right)\inv\Cont +\widetilde{L}WL\right)
\end{eqnarray}
\end{subequations}
where the stable matrix $W\in\stable^{(\ind+\outd)\times (\ind+\outd)}$ is such that $\det\left(\left(I-\Plant\Cont\right)\inv +\Plant \widetilde{L}WL\right)\neq 0$ and
$\det\left(\left(I-\Cont\Plant\right)\inv +\widetilde{L}WL\Plant\right)\neq 0.$

\item Denote
\begin{eqnarray*}%\label{eq:ErrorTM}
\widetilde{M} :=\begin{bmatrix}
-\left(I-\Plant\Cont\right)\inv \Plant\;\; & \left(I-\Plant\Cont\right)\inv
\end{bmatrix}, %\in\fractions{\stable}^{n\times (m+n)},
%\end{eqnarray*}
\quad \text{and}\quad
%\begin{eqnarray*}\label{eq:ErrorTM0}
M :=\begin{bmatrix}
\left(I-\Cont\Plant\right)\inv \\
\Plant\left(I-\Cont\Plant\right)\inv 
\end{bmatrix}. %\in\fractions{\stable}^{(m+n)\times m}.
\end{eqnarray*}
All plants that $\Cont$ stabilizes are parametrized by
\begin{subequations}
\begin{eqnarray}\label{eq:AllStabilizedPlants}
\Plant(X) & =\left(\Plant\left(I-\Cont\Plant\right)\inv +\widetilde{M}XM\right)
%\notag\\
%&\qquad \times
\left(\left(I-\Cont\Plant\right)\inv +\Cont \widetilde{M}XM\right)\inv \\
& =\left(\left(I-\Plant\Cont\right)\inv +\widetilde{M}XM\Cont\right)\inv %\notag\\
%&\qquad \times
\left(\left(I-\Plant\Cont\right)\inv\Plant +\widetilde{M}XM\right)
\end{eqnarray}
\end{subequations}
where the stable matrix $X\in\stable^{(\ind+\outd)\times (\ind+\outd)}$ is such that $\det\left(\left(I-\Cont\Plant\right)\inv +\Cont \widetilde{M}XM\right)\neq 0$ and $\det\left(\left(I-\Plant\Cont\right)\inv +\widetilde{M}XM\Cont\right)\neq 0.$
\end{enumerate}
\end{theorem}

Next two lemmas are given for later use. The first one is Theorem 5.3.10 of Vidyasagar\cite{Vidyasagar}. % which is used in the two stage controller design proposed in Section \ref{sec:Solv}. 
The original proof of the theorem uses coprime factorizations, but it can easily be extended to the more general setting of this paper.
% This is done in the next lemma.
%, we give the result here for the ease of reference and indicate the changes required in the original proof.

\begin{lemma}\label{lem:TwoStageStab}
If $\Cont_s$ stabilizes $\Plant$ and $\Cont_0$ stabilizes $\Plant_0=\Plant(I-\Cont_s\Plant)\inv$, then $\Cont_s+\Cont_0$ stabilizes $\Plant$.
\end{lemma}
\begin{proof}
The proof of the lemma is similar to the original proof by Vidyasagar \cite{Vidyasagar}. The only change required is to replace the arguments using coprime factorizations that show stability of $(I-(\Cont_s+\Cont_0)\Plant)\inv\Cont_s$. To this end, observe that $(I-\Cont_0\Plant_0)\inv=I+\Cont_0(I-\Plant_0\Cont_0)\inv\Plant_0$, so
\begin{eqnarray*}
(I-(C_s+C_0)\Plant)\inv \Cont_s & =& (I-\Cont_s\Plant)\inv(I-\Cont_0\Plant_0)\inv \Cont_s\\
& =& (I-\Cont_s\Plant)\inv \Cont_s+(I-\Cont_s\Plant)\inv\Cont_0(I-\Plant_0\Cont_0)\inv\Plant_0\Cont_s.
\end{eqnarray*}
This shows the claim since $(I-\Cont_s\Plant)\inv\Cont_s, \Cont_0(I-\Plant_0\Cont_0)\inv, \Plant_0\Cont_s\in\matrices{\stable}$ by the assumptions that $\Cont_s$ stabilizes $\Plant$ and $\Cont_0$ stabilizes $\Plant_0$. 
\end{proof}

\begin{lemma}\label{lem:exInvA}
%\begin{enumerate}
%\item 
%If $N\in\stable^{\outd\times\ind}$ and $D\in\stable^{\ind\times\ind}$ are such that $\det(D)\neq 0$ and $I=YN+XD$ for some $Y,X\in\matrices{\stable}$, then there exists $\widetilde{Y},\widetilde{X}\in\matrices{\stable}$ such that $I=\widetilde{Y}N+\widetilde{X}D$ and $\widetilde{X}\neq 0$.

%\item 
If $X,Y,Z\in\matrices{\stable}$ and $\gen\in\fractions{\stable}$ are such that $I=\gen X-YZ\in\fractions{\stable}^{n\times n}$, then there exist $\widetilde{X}, \widetilde{Z}\in\matrices{\stable}$ such that $\widetilde{X}$ is invertible over $\fractions{\stable}$ and $I=\gen\widetilde{X}-Y\widetilde{Z}$.
%\end{enumerate}
\end{lemma}

\begin{proof}
%In order to prove the first item, we define a plant $\Plant:=ND\inv\in\fractions{\stable}^{\outd\times\ind}$. By \cite[Theorem 1]{Mori2002} $\Plant_0:=\begin{bmatrix}\Plant^T &  0^{\ind\times\ind}\end{bmatrix}^T$ has a right coprime factorization $\Plant_0=N_rD_r\inv$ and a left coprime factorizations $\Plant_0=D\inv_lN_l$ such that $N_l Y_0+D_lX_0=I$ for some $Y_0,X_0\in\matrices{\stable}$. By \cite[Lemma 3.2]{Vidyasagaretal1982}, any $\Cont_0:=(Y_0-D_rS)(X_0+N_rS)\inv$ where $S\in\matrices{\stable}$ stabilizes $\Plant_0$ if $\det(X_0+N_rS)\neq 0$, i.e. $X_0+N_rS$ is invertible over $\fractions{\stable}^{\outd\times\outd}$...?

%Next we prove the second item. 
The columns of $X$ and $YZ$ together span $\fractions{\stable}^{n}$. Thus, there exists a basis $\{x_1,\ldots,x_k,h_{k+1},\ldots,h_n\}$ of $\fractions{\stable}^{n}$ where $x_1,\ldots,x_k$ are columns of $X$ forming the basis of its column space and $h_{k+1},\ldots,h_n$ are columns of $YZ$. For notational simplicity assume that $x_1,\ldots,x_k$ are the first $k$ columns of $X$. Denote $\gen=\frac{n}{d}$ where $n,d\in\stable$. Choose a matrix $M$ that selects the columns $h_{k+1},\ldots,h_{n}$ of $YZ$ so that
$$
X+dYZM=\begin{bmatrix}
x_1,\ldots,x_k,x_{k+1}+dh_{k+1},\ldots,x_n+dh_n
\end{bmatrix}\in\matrices{\stable}.
$$
The columns of $X+dYZM$ are linearly independent and consequently it is invertible over $\fractions{\stable}^{n\times n}$. Furthermore,
\begin{eqnarray*}
I = \gen (X+dYZM)-Y(nZM+Z)
\end{eqnarray*}
which shows the claim.
\end{proof}

Causality is the natural constraint that a system cannot depend on the future inputs. This can be considered in the chosen framework by following the approach by Vidyasagar, Schneider and Francis\cite{Vidyasagaretal1982}. To this end, a prime ideal $\primeideal$ of $\stable$ is fixed for the remaining part of this paper, and it is assumed that any causal transfer function has a factorization whose denominator is not in $\primeideal$. This leads to the following definition.

\begin{definition}
\begin{enumerate}\label{def:Causality}
\item A transfer function $p\in\fractions{\stable}$ is \emph{causal} if it has a factorization $p=\frac{n}{d}$ such that $n\in\stable$ and $d\in\stable\setminus\primeideal$.
		
\item A transfer function $p\in\fractions{\stable}$ is \emph{strictly causal} if it has a factorization $p=\frac{n}{d}$ such that $n\in\primeideal$ and $d\in\stable\setminus\primeideal$.

\item A transfer matrix $\Plant\in\matrices{\fractions{\stable}}$ is \emph{(strictly) causal} if all of its elements are (strictly) causal.

\item A square transfer matrix $\Plant\in\matrices{\fractions{\stable}}$ is \emph{$\primeideal$-nonsingular} if its determinant is in $\stable\setminus\primeideal$.
\end{enumerate}
\end{definition}

It is natural to assume that the plants and controllers discussed are causal. However, the general framework chosen allows a causal plant to have non-causal stabilizing controllers, i.e., the parametrized stabilizing controllers or plants stabilized by a controller in Theorem \ref{thm:Stability} may be non-causal. The approach adopted here is the same as by Mori and Abe\cite{MoriAbe2001}, i.e., non-causal plants and controllers are allowed in the general theoretical framework and causality issues are discussed separately.
%The consequences of this will be discussed in detail later in the text. 
To the best of author's knowledge, the most comprehensive results concerning the existence of causal controllers within the general framework are those of Propositions 6.1 and 6.2 in Mori and Abel\cite{MoriAbe2001}. These results are recalled in the following theorem for later use. 

\begin{theorem}\label{thm:ExistenceOfCausalControllers}
\begin{enumerate}
\item If a causal plant is stabilizable, it has a causal stabilizing controller.

\item All stabilizing controllers of a strictly causal plant are causal.
\end{enumerate}
\end{theorem}

\begin{remark}\label{rem:AlternativeCausality}
Results similar to those stated in Theorem \ref{thm:ExistenceOfCausalControllers} are given by Vidyasagar et al.\cite{Vidyasagaretal1982} and Sule\cite{Sule1994}. In particular, Vidyasagar et al. give the second item of the theorem assuming that the plant has both left and right coprime factorizations, but dropping out the requirement that the ideal $\primeideal$ is prime.
\end{remark}

Before proceeding to the problem formulation, a lemma concerning weakly coprime factorizations of causal transfer functions is given. It states that the denominator of any weakly coprime factorization of a causal transfer function is in $\stable\setminus\primeideal$. An arbitrary factorization does not have this property. E.g., given a factorization $\frac{n}{d}$ and $k\in\primeideal$, then $\frac{nk}{dk}$ is a factorization such that $dk\in\primeideal$.

\begin{lemma}\label{lem:WCoprimeFactCausality}
If $0\neq p\in\fractions{\stable}$ is a causal transfer function and $\frac{n}{d}$ its weakly coprime factorization, then $d\in\stable\setminus\primeideal$.
\end{lemma}
\begin{proof}
Since $p$ is causal, it has a factorization $\frac{n_0}{d_0}$ where $n_0\in\stable$ and $d_0\in\stable\setminus\primeideal$. The numerator $n\neq 0$ since $p\neq 0$, so $\frac{n_0}{n}n=n_0\in\stable$ and $\frac{n_0}{n}d=d_0\in\stable$. This implies that $\frac{n_0}{n}\in\stable$ since $\frac{n}{d}$ was assumed to be a weakly coprime factorization. The claim follows since $\frac{n_0}{n}d=d_0\in\stable\setminus\primeideal$ and all factors of an element in $\stable\setminus\primeideal$ belong to that set as well\cite{MoriAbe2001}.

 \end{proof}

\section{Problem formulation}
\label{sec:Problem}

Consider the control configuration of Fig. \ref{fig:Closedloop} where $\Plant\in\fractions{\stable}^{n\times m}$ and $\Cont\in\fractions{\stable}^{m\times n}$. The reference signals are assumed to be generated by a fixed signal generator $\Gen_r\in\fractions{\stable}^{n\times q}$, i.e. they are of the form $\reference=\Gen_r\measurement_0$ where $\measurement_0\in\stable^{q}$. This article is focused on regulation, so it is assumed that the disturbance signals contain only unstable dynamics that are already present in the signal generator. In other words, it is assumed that the disturbance signals are of the form $\disturbance=\Gen_d\disturbance_0$ where the vector $\disturbance_0\in\stable^{q}$ and $\Gen_d=Q\Gen_r\in\fractions{\stable}^{m\times q}$ for some fixed matrix $Q\in\stable^{m\times n}$. It is assumed throughout this paper that the given nominal plant $\Plant$ and the signal generator $\Gen_r$ are causal.

\begin{definition}\label{def:RegDRej}
\begin{enumerate}
\item The controller $\Cont$ %\in\fractions{\stable}^{m\times n}$ 
\emph{regulates} $\Plant$ %\in\fractions{\stableC}^{n\times m}$ 
if %\linebreak
%\begin{eqnarray*}
$
\left(I-\Plant\Cont\right)\inv \Gen_r\measurement_0\in \stable^{\outd}
$
%\end{eqnarray*}
for all $\measurement_0\in\stable^{q}$. Furthermore, 
a controller $\Cont$ \emph{robustly regulates} $\Plant$ if 
\begin{enumerate}[i)]
\item it stabilizes $\Plant$, and
\item regulates every plant it stabilizes.
\end{enumerate}

\item The controller $\Cont$ is \emph{causally robustly regulating} for $\Plant$ if
\begin{enumerate}[i)]
\item it stabilizes $\Plant$, and
\item regulates every causal plant it stabilizes.
\end{enumerate}

\item A controller $\Cont$ is \emph{disturbance rejecting} for $\Plant$ if
$
%\begin{eqnarray*}
\left(I-\Plant\Cont\right)\inv \Plant \Gen_d\disturbance_0\in \stable^{\outd}
%\end{eqnarray*}
$
for all $\disturbance_0\in\stable^{q}$. Furthermore,
a controller $\Cont$ is \emph{robustly disturbance rejecting} for $\Plant$ if 
\begin{enumerate}[i)]
\item it stabilizes $\Plant$, and
\item is disturbance rejecting for every plant it stabilizes.
\end{enumerate}
\end{enumerate}
\end{definition}

The problem of finding a controller $\Cont$ that (causally) robustly regulates a given nominal plant $\Plant$ is called \emph{the (causal) robust regulation problem.}

\begin{remark}
By Remark \ref{rem:IOStability}, the time domain interpretation of regulation in Definition \ref{def:RegDRej} is that the error between the output and the reference signal converges asymptotically to zero, i.e., $\| y(t)-\reference(t)\|\to 0$ as $t\to\infty$, for all reference signals. %generated by the signal generator.
%Actually, requiring $\ref_0\in\stable^{q\times 1}$ is not necessary as long as 
\end{remark}

\begin{remark}
Note that no topology is needed in the above definitions, and robustness of regulation simply means that stability of the closed loop implies regulation. This definition is justified since stability of the closed loop is a robust property under very general assumptions in the algebraic framework\cite{BallSasane2012,Vidyasagaretal1982}. This means that a controller solving the robust regulation problem regulates all plants sufficiently near the given plant $\Plant$. 
\end{remark}

\begin{remark}\label{rem:ControllerCausality}
	It is not assumed here that the controller $\Cont$ is causal. This is because the parametrized controllers in Theorem \ref{thm:Stability} may be non-causal and some of the existence results are based on the theorem. However, for strictly causal plants such an assumption is unnecessary by Theorem \ref{thm:ExistenceOfCausalControllers}. For causal plants the situation is more complicated and will be discussed later.
\end{remark}

\begin{remark}\label{rem:RRPvsCRRP}
	The plants appearing in applications can be assumed to be causal, so the causal robust regulation problem is the more natural one of the two problems posed here. It is obvious that all robustly regulating controllers are also causally robustly regulating, so the more general theory involving also non-causal plants will produce controllers with the desired properties and is worth studying. On the other hand, a controller may stabilize also non-causal plants, so the robust regulation problem may lead to slightly too strong requirements for the controller. It will be shown in the next section that a strictly causal controller that is causally robustly
regulating is also robustly regulating.%, see Corollary \ref{cor:IMPStrictlyCausal}.
\end{remark}

%%%%%%%%%%%%%%%%%%%%%%%%%%%%%%%%%%%%%%%%%%%%%%%%%%%%%%%%%%%%%
%% ==> Characterizations of robust controllers:
\section{The internal model principle}\label{sec:IMP}

The aim of this section is to extend the classical frequency domain formulation given by Vidyasagar \cite{Vidyasagar} using fractional representations.
%In Section \ref{sec:Simplification}, this result is confirmed to be equivalent to the classical formulation provided that the plant and the signal generator have coprime factorizations.
The next theorem is a major step towards that goal. It shows that the unstable dynamics generated by any element of the signal generator $\Gen_r$ must be blocked by every element of the sensitivity and load disturbance sensitivity matrices. Thus, it does not matter if some unstable dynamics appear, e.g., only in the first element of any reference signal. For systems described by rational matrices this corresponds to the fact that the closed loop transfer matrix must vanish completely at each pole of the signal generator located in the closed right half plane $\overline{\C}_+$. \cite{FrancisWonham1975b}

\begin{theorem}\label{thm:GenTFDiv}
A stabilizing controller $\Cont$ is robustly regulating for $\Plant$ with the signal generator $\Gen_r=(\gen_{ij})$ if and only if for all $i\in\List{\outd}$ and $j\in\List{q}$
\begin{eqnarray}\label{eq:GenTFDiv}
\gen_{ij}\begin{bmatrix}
-(I-\Plant\Cont)\inv\Plant & (I-\Plant\Cont)\inv\end{bmatrix}\in\matrices{\stable}.
\end{eqnarray}
%\begin{eqnarray}\label{eq:GenTFDiv}
%\gen_{ij}\begin{bmatrix}
%(I-\Plant\Cont)\inv & (I-\Plant\Cont)\inv \Plant
%\end{bmatrix}\in\matrices{\stable},
%\end{eqnarray}
\end{theorem}
\begin{proof}
Denote $M=\begin{bmatrix}
((I-\Cont\Plant)\inv)^T & (\Plant(I-\Cont\Plant)\inv)^T
\end{bmatrix}^T$, $U=(I-\Plant\Cont)\inv$, and $\widetilde{M}=\begin{bmatrix}
-U\Plant & U\end{bmatrix}$. These matrices are stable since $\Cont$ stabilizes $\Plant$. Since $\Cont$ stabilizes $\Plant$, the second part of Theorem \ref{thm:Stability} gives the parametrization of plants $\Plant(X)$ the controller $\Cont$ stabilizes and
\begin{subequations}
\label{eq:RegStabPert}
\begin{eqnarray}
(I-\Plant(X)\Cont)\inv \Gen_r&=&\left(U +\widetilde{M}X M\Cont -\left(U \Plant+\widetilde{M}X M\right)\Cont\right)\inv\left(U +\widetilde{M}X M\Cont\right)\Gen_r\\
&=& \left(U -U \Plant\Cont\right)\inv \left(U +\widetilde{M}X M\Cont\right)\Gen_r\\
& = &U\Gen_r +\widetilde{M}X M\Cont\Gen_r
\end{eqnarray}
\end{subequations}
where $\det\left(\left(I-\Plant\Cont\right)\inv +\widetilde{M}XM\Cont\right)\neq 0$.

The sufficiency is shown first. Assume that \eqref{eq:GenTFDiv} holds.
%, i.e., $\gen_{ij}\widetilde{M}\in\matrices{\stable}$. This also implies that $\gen_{ij}U\in\matrices{\stable}$. 
Observe that the reference signal $\reference=\Gen_r \measurement_0$ with an arbitrary stable vector $\measurement_0$ of suitable size
%and the disturbance signal $\disturbance=\Gen_d \disturbance_0$ 
can be written in the form
%\begin{eqnarray*}
$\reference=\sum_{i,j}\gen_{ij}\measurement_{ij} $ 
%\qquad\text{and}\qquad \disturbance=\sum_{i,j}\gen_{ij}\disturbance_{ij},
%\end{eqnarray*}
where $\measurement_{ij}$ %and $\disturbance_{ij}$ are stable vectors.
are stable vectors. Substituting this into \eqref{eq:RegStabPert} yields \begin{eqnarray*}
(I-\Plant(X)\Cont)\inv \Gen_r y_0=\sum_{i,j}\gen_{ij}U y_{ij}+\sum_{i,j}\gen_{ij}\widetilde{M}XMCy_{ij}.
\end{eqnarray*}
This vector is stable since $\gen_{ij}\widetilde{M},\gen_{ij}U\in\matrices{\stable}$ by \eqref{eq:GenTFDiv} and $M\Cont\in\matrices{\stable}$ since $\Cont$ stabilizes $\Plant$.
Thus, $\Cont$ is robustly regulating.

The necessity is shown next. Assume that $\Cont$ robustly regulates $\Plant$. Since $\Cont$ regulates all the plants it stabilizes, the matrix in \eqref{eq:RegStabPert} is stable for all matrices $X$. Choosing $X=0$ yields $U\Gen_r\in\matrices{\stable}$. This and \eqref{eq:RegStabPert} imply that $\widetilde{M}X M\Cont\Gen_r\in\matrices{\stable}$. In particular,
\begin{eqnarray*}
 \widetilde{M}\begin{bmatrix}
0 & 0\\
0 & X_0
\end{bmatrix} M\Cont\Gen_r &  = &
\begin{bmatrix}
-U\Plant & U
\end{bmatrix}
\begin{bmatrix}
0 & 0\\
0 & X_0
\end{bmatrix}
\begin{bmatrix}
(I-\Cont\Plant)\inv \\ \Plant(I-\Cont\Plant)\inv
\end{bmatrix}
\Cont\Gen_r\\
& = &U X_0 \Plant(I-\Cont\Plant)\inv\Cont\Gen_r\\
& = &U X_0 (I-\Plant\Cont)\inv\Plant\Cont\Gen_r\\
& = &U X_0 (U-I)\Gen_r\in\matrices{\stable}
\end{eqnarray*}
if $\det(U+U X_0 (U-I))\neq 0$. Since $U\Gen_r\in\matrices{\stable}$, it follows that
$UX_0\Gen_r\in\matrices{\stable}$.

Choose $X_0=\mathbf{e}_k\mathbf{e}_i^T$ where $\mathbf{e}_h$ is the $h$th narural basis vector of $\fractions{\stable}^{n}$. If $\det(U+U X_0 (U-I))\neq 0$ for some $i,k\in\{1,\ldots,n\}$, then $U\mathbf{e}_k \mathbf{e}_i \Gen_r\in\matrices{\stable}$. If for some some $i,k\in\{1,\ldots,n\}$
$$
\det(U+U \mathbf{e}_k\mathbf{e}_i^T (U-I))=\det(U)\det(I+\mathbf{e}_k\mathbf{e}_i^T (U-I))= 0,$$
then $\det(I+\mathbf{e}_k\mathbf{e}_i^T (U-I))= 0$ since $\det(U)\neq 0$ by the stability of the closed loop system. Write $K:=\mathbf{e}_k\mathbf{e}_i^T (U-I)=\mathbf{g}\mathbf{h}^T$  where $\mathbf{g},\mathbf{h}\in \fractions{\stable}^{n}$. This is possible since $K$ is a rank one matrix. By the matrix determinant lemma, $\det(I+K)=1+\mathbf{h}^T\mathbf{g}=0$. This can happen only if $\mathbf{h}^T\mathbf{g}=-1$, so choosing $X_0=2\mathbf{e}_k\mathbf{e}_i^T$ one has $\det(U+U X_0 (U-I))\neq 0$ and $U\mathbf{e}_k \mathbf{e}_i \Gen_r\in\matrices{\stable}$.

It has been shown that $U\mathbf{e}_k \mathbf{e}_i \Gen_r\in\matrices{\stable}$ for all $k,i\in\List{\outd}$. This means that the $k$th column of $U$ multiplied by any element in the $i$th row of $\Gen_r$ must be stable. Thus,
$
\gen_{ij}(I-\Plant\Cont)\inv\in\matrices\stable
$ 
for all possible $i\in\{1,\ldots,n\}$ and $j\in\{1,\ldots,q\}$.

In order to complete the proof, choose $X=\begin{bmatrix}
0 & X_0\\0 & 0
\end{bmatrix}$ in \eqref{eq:RegStabPert}. Similar arguments as above show that 
%\begin{eqnarray*}
$
\gen_{ij}(I-\Plant\Cont)\inv\Plant \in\matrices{\stable},
$ 
%\end{eqnarray*}
for all $i\in\List{\outd}$ and $j\in\List{q}$ so \eqref{eq:GenTFDiv} follows. 
\end{proof}

\begin{remark}\label{rem:RORPandRRC}
Similar aguments as in the sufficiency part of the above proof show that \eqref{eq:GenTFDiv} implies that 
$
%\begin{eqnarray*}
\left(I-\Plant(X)\Cont\right)\inv \Plant(X) \Gen_d\disturbance_0\in \matrices{\stable}
%\end{eqnarray*}
$ 
for all suitable $X$. This means that \eqref{eq:GenTFDiv} implies robust disturbance rejection as well. Thus, all robustly regulating controllers are also robustly disturbance rejecting. On the other hand, it was observed in Example 5.4 of Laakkonen and Pohjolainen\cite{LaakkonenPohjolainen2015} that there exist robustly disturbance rejecting controllers that are not robustly regulating.
\end{remark}

\begin{example}
The condition \eqref{eq:GenTFDiv} of Theorem \ref{thm:GenTFDiv} is equivalent of $\Cont$ being regulating and disturbance rejecting for $\Plant$ with all reference and disturbance signals of the form $\reference=\gen_{ij}\mathbf{e}_k$ and $\disturbance=\gen_{ij} \mathbf{e}'_h$ where $\mathbf{e}_k$ and $\mathbf{e}'_h$ are arbitrary natural basis vector of $\fractions{\stable}^\outd$ and $\fractions{\stable}^\ind$, respectively. This can be used to test if a controller is robustly regulating. For example, if one wants to find out if a controller $\Cont$ achieves robust regulation for the single reference signal $\reference=(0,\gen,0,0,\ldots,0)\in \fractions{\stable}^\outd$, then one needs to test if $\Cont$ is regulating for all reference signals of the form $\gen\mathbf{e}_k$ with $k\in\List{\outd}$ and disturbance rejecting for all disturbance signals $\gen_{ij}\mathbf{e}'_h$ with $h\in\List{\ind}$.
\end{example}

The next theorem is the first main result of this paper. It is a reformulation of the famous internal model principle using no coprime factorizations. 
It states that all the unstable dynamics produced by the signal generator must be built into the controller as an internal model in order to make it robustly regulating. It generalizes Theorem 3.2 of Laakkonen and Quadrat \cite{LaakkonenQuadrat2017} from SISO %single-input single-output 
systems to MIMO %multi-input multi-output 
systems.
In Section \ref{sec:Simplification}, the new formulation is confirmed to be equivalent to the classical one of Vidyasagar \cite{Vidyasagar}  
when coprime factorizations exist.
%provided that the plant and the signal generator have coprime factorizations. 

\begin{theorem}\label{thm:IMP}
Denote $\Gen_r=(\gen_{ij})$. Controller $\Cont$ solves the robust regulation problem for $\Plant$ if and only if it stabilizes $\Plant$ and for all $i\in\List{\outd}$ and $j\in\List{q}$ there exist $A_{ij},B_{ij}\in\matrices{\stable}$ such that 
\begin{eqnarray}\label{eq:IMP}
\gen_{ij} I=A_{ij}+B_{ij}\Cont.
\end{eqnarray}
\end{theorem}

\begin{proof}
Denote $\widetilde{M}=\begin{bmatrix}
-\left(I-\Plant\Cont\right)\inv \Plant\;\; & \left(I-\Plant\Cont\right)\inv
\end{bmatrix}$. If $\gen_{ij}I=A_{ij}+B_{ij}\Cont$, it follows that 
%\begin{eqnarray*}
$
\gen_{ij}\widetilde{M}
 =A_{ij}\widetilde{M}+B_{ij}\Cont\widetilde{M}\in\matrices{\stable}
$. 
%\end{eqnarray*}
The matrices $\widetilde{M}$ and $\Cont\widetilde{M}$ on the right hand side of the equation are stable since $\Cont$ stabilizes $\Plant$. The controller $C$ is robustly regulating by Theorem \ref{thm:GenTFDiv}. This shows sufficiency. The necessity follows since
$$
\gen_{ij}I=\gen_{ij}(I-\Plant\Cont)\inv (I-\Plant\Cont)=\gen_{ij}(I-\Plant\Cont)\inv-\gen_{ij}(I-\Plant\Cont)\inv \Plant\Cont
$$
where $A_{ij}=\gen_{ij}(I-\Plant\Cont)\inv$ and $B_{ij}=-\gen_{ij}(I-\Plant\Cont)\inv\Plant$ are stable matrices by Theorem \ref{thm:GenTFDiv} when $\Cont$ is robustly regulating.
\end{proof}

As mentioned in Remark \ref{rem:RRPvsCRRP}, the definition requires a robustly regulating controller to regulate even the non-causal plants it stabilizes. By changing the roles of the plant and controller in Theorem \ref{thm:ExistenceOfCausalControllers}, one has that a strictly causal controller stabilizes only causal plants. Combining it with the above theorem, the following corollary is obtained.

\begin{corollary}\label{cor:IMPStrictlyCausal}
A strictly causal stabilizing controller $\Cont$ is causally robustly regulating for $\Plant$ with the signal generator $\Gen_r=(\gen_{ij})$ if and only if it satisfies the two equivalent conditions \eqref{eq:GenTFDiv} and \eqref{eq:IMP} for all elements $\gen_{ij}$.
\end{corollary}

\begin{remark}
Theorems \ref{thm:GenTFDiv} and \ref{thm:IMP} are necessary and sufficient conditions for a stabilizing controller to be robustly regulating. Thus, they can be thought as alternative characterizations of the internal model principle. However, only the latter one directly deals with the structure of the controller whereas the first one states a property of the closed loop system. Therefore the latter one is considered as the internal model principle in this article.
\end{remark}

%\begin{lemma}\label{lem:IMP2}
%Denote $\Gen=(\gen_{ij})$ and let $\Cont$ stabilize $\Plant$. If the controller $\Cont$ is robustly regulating for $\Plant$ then for all $i\in\List{\outd}$ and $j\in\List{q}$ there exist $A_{ij},B_{ij}\in\matrices{\stable}$ such that 
%$ %\begin{eqnarray}\label{eqn:RegulationSolvability}
%\gen_{ij} I=A_{ij}+B_{ij}\Cont.
%\end{eqnarray}
%$
%\end{lemma}
%\begin{proof}
%The proof is completed by choosing the stable matrices $A_{ij}=\gen_{ij}(I-\Plant\Cont)\inv$ and $B_{ij}=\gen_{ij}(I-\Plant\Cont)\inv\Plant$ and observing that
%\begin{eqnarray*}
%\gen_{ij} I=\gen_{ij} (I-\Plant\Cont)\inv(I-\Plant\Cont)=A_{ij}+B_{ij}\Cont.
%\end{eqnarray*}

%If $\det((I-\Plant\Cont)\inv+(I-\Plant\Cont)\inv A_{22}\Plant (I-\Plant\Cont)\inv\Cont)=0$, then...
%\end{proof}

\begin{example}\label{exa:IMPSISO}
For SISO plants, Theorem \ref{thm:IMP} takes the form $\langle \Gen_r\rangle\subseteq \langle 1,\Cont\rangle$.\cite{LaakkonenQuadrat2015} The inclusion indicates that the signals generated by the generator can be divided into a stable part and an unstable part generated by the controller. This makes sense, since only the unstable dynamics need to be regulated by $\Cont$.
\end{example}

\begin{example}
Consider linear systems described by ordinary differential equations with $n$ inputs and outputs. The natural choice for the ring of stable transfer functions is $\stable=\RHinf$ and then $\fractions{\RHinf}=\R(s)$. The causal plants in $\R(s)$ are the proper rational matrices. Such a matrix is stable if and only if it has no poles in the closed right half-plane $\overline{\C}_+$, so the unstable dynamics of the reference signals are characterized by the poles in $\overline{\C}_+$.

Consider a signal generator $\Gen_r=(\gen_{ij})\in\fractions{\RHinf}^{n\times n}$. The condition \eqref{eq:IMP} means that if $\gen_{ij}$ has a pole of order $k$ at $z\in\overline{\C}_+$, then any robustly regulating controller $\Cont$ must have a pole of the same or higher order at $z$. Furthermore, one can write $\Cont=U\Lambda V$ where $\Lambda=\diag{h_1,\ldots,h_n}\in \fractions{\RHinf}^{n\times n}$ is the Smith-McMillan form of $\Cont$ and $U$ and $V$ are invertible over $\RHinf$.\cite{Vidyasagar} Then \eqref{eq:IMP} holds if and only if each of the $n$ diagonal elements $h_j$ in the Smith-McMillan form have a pole of order $k$ or higher at $z$. 
%Each of the $n$ diagonal elements of $\Lambda$ having the pole 
This corresponds to the well-known fact that a robustly regulating controller must contain a $n$-folded copy of any unstable dynamics of reference signals. \cite{PaunonenPohjolainen2010,FrancisWonham1975a}
%Since $\RHinf$ is a principal ideal domain $\langle \gen_{ij} | 1\leq i\leq n,\, 1\leq j\leq q\rangle=\langle \gen\rangle$ where $\gen(s)=\frac{n(s)}{d(s)}\in \RHinf$. Furthermore, $d(s)$ has zeros at each pole of the $\gen_{ij}(s)$ in the closed half-plane $\overline{\C}_+$ and the order of the zero is the maximal order of the corresponding poles appearing in $\gen_{ij}(s)$. Thus, the unstable dynamics of the reference signals are characterized by the poles of the signal generator. 
\end{example}

\section{Simplification of the internal model}
\label{sec:Simplification}

The results of the previous section revealed that the unstable dynamics of each element in the signal generator must be included into a robustly regulating controller. The next theorem shows that the internal model %to be constructed into the controller can be 
is characterized by the fractional ideal generated by the %non-zero
elements of the signal generator. The theorem provides a way to characterize the internal model in a compact form as illustrated by Example \ref{exa:Simplification}.

%Theorem \ref{thm:IMP} shows that the unstable dynamics of each element element $\gen_{ij}$ of the signal generator must be built into every element of a robustly regulating controller. Checking the condition \eqref{eq:IMP} for every $\gen_{ij}$ separately is not always needed. The overall instability generated by $\Gen_r$ is characterized by the fractional ideal generated by the elements of $\Gen_r$ and is often generated by a smaller set of elements. The following corollary makes this statement precise.

\begin{theorem}\label{thm:Simplification}
Let $\Cont$ stabilize $\Plant$. Consider $\Gen_r=(\gen_{ij})$ and the fractional ideal $J=\langle \gen_{ij} | 1\leq i\leq n, 1\leq j\leq q\rangle$.
\begin{enumerate}
\item If $J\subseteq \langle f_1,\ldots,f_k\rangle$ and there exist $A_l$ and $B_l$ such that $f_l I=A_l+B_l \Cont$ for all $l=1,\ldots, k$, then $\Cont$ is robustly regulating. 
\item If $\langle f_1,\ldots,f_k\rangle\subseteq J$ and $\Cont$ is robustly regulating, then there exist $A_l$ and $B_l$ such that $f_l I=A_l+B_l \Cont$ for all $l=1,\ldots, k$.
\end{enumerate}
\end{theorem}
\begin{proof}
Only the first part is proved. The second part can be proved using similar arguments. It is assumed that $J\subseteq \langle f_1,\ldots,f_k\rangle$ and that there exist $A_l$ and $B_l$ such that $f_l I=A_l+B_l \Cont$ for all $l=1,\ldots, k$. Now $\gen_{ij}\in \langle f_1,\ldots,f_k\rangle$ or equivalently 
%\begin{eqnarray*}
$
\gen_{ij}=a_1 f_1+\cdots +a_k f_k
$ 
%\end{eqnarray*}
for some $a_1,\ldots, a_k\in \stable$. Consequently,
\begin{eqnarray*}
\gen_{ij} I & =\sum_{l=1}^k a_l f_l I = \sum_{l=1}^k a_l(A_l+B_l \Cont) =\left(\sum_{l=1}^k a_l A_l\right)+\left(\sum_{l=1}^{k}a_l B_l\right)\Cont
\end{eqnarray*}
where $\sum_{l=1}^k a_l A_l$ and $\sum_{l=1}^{k}a_l B_l$ are stable matrices. Since $\gen_{ij}$ is an arbitrary element of $\Gen_r$, the result follows by Theorem \ref{thm:IMP}. 
\end{proof}

The above theorem shows that the instability generated by $\Gen_r=(\gen_{ij})$ is captured by the fractional ideal $J$ generated by the elements $\gen_{ij}$. In particular, if $J$ is principal and %there exists an element $\gen\in\fractions{\stable}$ such that
$J=\langle\gen\rangle$ where $\gen\in\fractions{\stable}$, then a stabilizing controller is robustly regulating if and only if there exist stable $A$ and $B$ such that 
$
%\begin{eqnarray*}
\gen I=A+B\Cont
%\end{eqnarray*}
$. 
This in particular means that if $\stable$ is a Bezout domains, i.e., a domain where all finitely generated ideals are principal, the required internal model is always characterized by a single element of $\fractions{\stable}$.
%The algebraic structures where all finitely generated fractional ideals are principal are called Bezout domains. 
%Thus, if $\stable$ is a Bezout domain, the internal model can be characterized by a single element of $\fractions{\stable}$.

\begin{example}\label{exa:example1}
The integral domain $\RHinf$ is a principal ideal domain.\cite{Vidyasagar} Consequently, it is a Bezout domain and the internal model is always captured by a single rational function in the field $\fractions{\RHinf}=\R(s)$ if $\stable=\RHinf$.
Other common rings in systems theory, e.g., the Hardy space $\Hinfty$  or the convolution algebra $\mathcal{A}(\beta)$\cite{CallierDesoer1978}, are not typically Bezout. Consequently, there are signal generators generating instability that cannot be captured by any single fraction over the ring.
\end{example}

\begin{example}\label{exa:Simplification}
Choose $\stable=\Hinfty$ and consider the signal generator 
\begin{eqnarray*}
\Gen_r(s)=\begin{bmatrix}
\frac{1}{\e^{-s}-1} & 0\\
0 & \frac{1}{s}+\frac{1}{s^2+\pi^2}
\end{bmatrix}.
\end{eqnarray*}
Then the reference signals are of the form 
\begin{eqnarray*}
\reference(s)=\begin{bmatrix}
\frac{1}{\e^{-s}-1} & 0\\
0 & \frac{1}{s}+\frac{1}{s^2+\pi^2}
\end{bmatrix}\begin{bmatrix}
a(s)\\
b(s)
\end{bmatrix}=\begin{bmatrix}
\frac{1}{\e^{-s}-1}a(s) \\ \left(\frac{1}{s}+\frac{1}{s^2+\pi^2}\right)b(s) 
\end{bmatrix}
\end{eqnarray*}
where $a(s),b(s)\in\Hinfty$.
The inclusion $\langle \frac{(s+1)^2}{(\e^{-s}-1)(s^2+\pi^2)}\rangle\subseteq \langle \frac{1}{\e^{-s}-1}, \frac{1}{s}+\frac{1}{s^2+\pi^2}\rangle$ holds since
\begin{eqnarray*}
\frac{(s+1)^2}{(\e^{-s}-1)(s^2+\pi^2)} & =\frac{1}{\e^{-s}-1}\alpha(s)+\left(\frac{1}{s}+\frac{1}{s^2+\pi^2}\right)\beta(s)
\end{eqnarray*}
where
\begin{eqnarray*}
\beta(s)  &=& \frac{3\pi^2-1+s(\pi^2-3)}{2(s+1)}\in\Hinfty,\\%\frac{(4\pi^2-4)s+6\pi^2-\pi^4-1}{2(s+1)^2}\in\Hinfty,\\
\alpha(s)  & = & \frac{(s+1)^2}{s^2+\pi^2}-\beta(s)\left(\frac{\e^{-s}-1}{s}+\frac{\e^{-s}-1}{s^2+\pi^2}\right)\\& =&\frac{2(s+1)^3+(1-3\pi^2+s(3-\pi^2))(\e^{-s}-1)}{2(s+1)(s^2+\pi^2)}-\beta(s)\frac{\e^{-s}-1}{s}\in\Hinfty. %\frac{(s+1)^2}{s^2+\pi^2}-(e^{-s}-1)\left(\frac{1}{s}+\frac{1}{s^2+\pi^2}\right)\beta(s)\in\Hinfty.
\end{eqnarray*}
Note that $\alpha(s)$ is stable due to the pole-zero cancellations at $\pm\pi\i$ and $0$. Similarly, $$\langle \frac{1}{\e^{-s}-1}, \frac{1}{s}+\frac{1}{s^2+\pi^2}\rangle\subseteq\langle \frac{(s+1)^2}{(\e^{-s}-1)(s^2+\pi^2)}\rangle$$ since
\begin{eqnarray*}
&\frac{1}{\e^{-s}-1}& = \frac{(s+1)^2}{(\e^{-s}-1)(s^2+\pi^2)}\frac{s^2+\pi^2}{(s+1)^2},\\
&\frac{1}{s}+\frac{1}{s^2+\pi^2}&=\frac{(s+1)^2}{(\e^{-s}-1)(s^2+\pi^2)}\frac{(\e^{-s}-1)(s^2+\pi^2+s)}{s(s+1)^2}.
\end{eqnarray*}
%Basic computations show that 
Thus,
\begin{eqnarray}\label{eq:Simplification}
\left\langle \frac{(s+1)^2}{(\e^{-s}-1)(s^2+\pi^2)}\right\rangle= \left\langle \frac{1}{\e^{-s}-1}, \frac{1}{s}+\frac{1}{s^2+\pi^2}\right\rangle.
\end{eqnarray}
This means that the overall instability generated by the signal generator is captured by the single element $\gen(s)=\frac{(s+1)^2}{(\e^{-s}-1)(s^2+\pi^2)}$. This can be verified by observing that the pole locations and orders of $\gen(s)$ are exactly the same that appear in the non-zero elements of the signal generator. Note that both diagonal elements of $\Gen_r(s)$ have a first order pole at $s=0$, but this does not raise the order of the corresponding pole in $\gen(s)$. In fact, if multiple elements of the signal generator have a pole at the same location, only the maximal order over them matters.

According to Theorems \ref{thm:IMP} and \ref{thm:Simplification}, a stabilizing controller is robustly regulating if and only if there exist stable matrices $A,B\in\matrices{\Hinfty}$ such that $\gen I=A+B\Cont$ where $\gen=\frac{(s+1)^2}{(\e^{-s}-1)(s^2+\pi^2)}$. Stable matrices $A$ and $B$ cannot have poles in $\overline{\C}_+$, so the unstable poles of the reference signals must be found in %every non-zero element of 
$\Cont$. 

%By Corollary \ref{cor:PertubGen}, a controller remains robustly regulating for a perturbed signal generator $\Phi_r$ as long as $\Phi_r=\gen \Phi_0$ for some stable $\Phi_0$. This means that some of the unstable poles can be removed, but no new poles can appear.
\end{example}

\begin{theorem}\label{thm:IMPClassic}
Consider $\Gen_r=(\gen_{ij})$, let $\Cont$ stabilize $\Plant$ and assume that the fractional ideal $J=\langle \gen_{ij} | 1\leq i\leq n, 1\leq j\leq q\rangle$ is principal. Let $\gen\in\fractions{\stable}$ be such that $J=\langle\gen\rangle$. If $\gen=\frac{n}{d}$ is a weakly coprime factorization, then $\Cont$ is robustly regulating if and only if
there exist $A_0,B_0\in\matrices{\stable}$ such that
\begin{eqnarray}\label{eq:IMPClassic1}
d\inv I=A_0+B_0\Cont.
\end{eqnarray}
If in addition to the above assumptions $\Cont$ has a right coprime factorization $\Cont=N D^{-1}$,
then $\Cont$ is robustly regulating if and only if $D=d D_0$ for some $D_0\in\matrices{\stable}$.
\end{theorem}

%The proof of the above theorem shows that if the internal model is captured by a single element, then the internal model principle can be expressed by one equation. This idea of simplification of the internal model is discussed in the next section and the proof of the theorem is postponed until the end of that section.
\begin{proof}%[Proof of Theorem \ref{thm:IMPClassic}]
First, the controller $\Cont$ is shown to be robustly regulating if \eqref{eq:IMPClassic1} holds. 
Multiplying both sides of \eqref{eq:IMPClassic1} by $n$ implies 
%Theorem \ref{thm:Simplification} implies that $\Cont$ is robustly regulating if and only if for some stable $A$ and $B$
$
%\begin{eqnarray}\label{eq:IMPClassic2}
\gen I=A+B\Cont
%\end{eqnarray}
$. This shows that $\Cont$ is robustly regulating by Theorem \ref{thm:Simplification}.
%Multiplying both sides of \eqref{eq:IMPClassic1} by $n$ shows that \eqref{eq:IMPClassic1} implies \eqref{eq:IMPClassic2}.

Next it is shown that robust regulation implies \eqref{eq:IMPClassic1}. To this end, set $A_0=d\inv(I-\Plant\Cont)\inv$ and $B_0=-d\inv(I-\Plant\Cont)\inv\Plant$. The matrices $nA_0$ and $dA_0$ are stable since $C$ is robustly regulating and stabilizing.  Since $\gen=\frac{n}{d}$ is a weakly coprime factorization this implies that $A_0$ is stable. Similar arguments show that $B_0$ is stable. The claim follows since 
$
%\begin{eqnarray*}
d\inv I=d\inv(I-\Plant\Cont)\inv(I-\Plant\Cont)=A_0+B_0\Cont.
%\end{eqnarray*}
$

The second statement of the theorem is shown by proving that \eqref{eq:IMPClassic1} is equivalent to that $D=dD_0$ for some $D_0\in\matrices{\stable}$. It is assumed that $\Cont=N D^{-1}$ is a right coprime factorization, so there exist $X,Y\in\matrices{\stable}$ such that $XD-YN=I$. If \eqref{eq:IMPClassic1} holds, then 
$
%\begin{eqnarray*}
d\inv D=(A_0+B_0\Cont)D=A_0 D+B_0 N:=D_0\in\matrices{\stable}
%\end{eqnarray*}
$ 
or equivalently $D=dD_0$. The conclusion follows since, if $D=d D_0$, then
\begin{eqnarray*}
d\inv I=D_0 D\inv=D_0(XD-YN)D\inv=D_0 X -D_0 Y \Cont.
\end{eqnarray*}
\end{proof}

Any weakly right coprime factorization of a stabilizing controller is also a right coprime factorization.\cite{Quadrat2006b}
%\cite[Corollary 8]{Quadrat2006b}
Thus, it is sufficient to check that a given factorization of the controller is weakly coprime when applying the result of Theorem \ref{thm:IMPClassic}. Note that the assumption that $\gen$ of the above theorem has a coprime factorization does not follow by the existence of a weak coprime factorization \cite{LaakkonenQuadrat2017}
%\cite[Example 5.2]{LaakkonenQuadrat2017}
and cannot be done without restricting generality.

\begin{remark}\label{rem:SimpleInternalModel}
If 
$J=\langle \gen_{ij} | 1\leq i\leq n, 1\leq j\leq q\rangle$ is principal and its generator has a weakly coprime factorization
$\gen=\frac{n}{d}$, then the internal model to be built into a robustly regulating controller is characterized by the stable element $d$. By Theorem \ref{thm:Simplification}, $d$ is unique up to multiplication by a unit of $\stable$. In this sense, one has a minimal internal model. E.g., in Example \ref{exa:Simplification}, the stable element characterizing the minimal internal model is $d=\frac{(\e^{-s}-1)(s^2+\pi^2)}{(s+1)^2}$. By the first item of Theorem \ref{thm:Simplification}, one may choose $d^{-1}$ to be the internal model even if $n$ and $d$ are not coprime. However, then the internal model is not minimal since $d\inv$ produces stronger instability than $\Gen_r$ is able to generate, or in other words $J\subsetneq\langle d\inv\rangle$.

Provided that the stable element $d$ characterizing the minimal internal model exists, it must divide all elements of the denominator $D$ in a coprime factorization of the controller. By Theorems 7.8 and 7.9 of Lang\cite{Lang2002}, if $\stable$ is a Bezout domain, $d$ is the largest invariant factor of the denominator of the coprime factorization of $\Gen_r$. This shows that Theorem \ref{thm:IMPClassic} corresponds to Lemma 7.5.8 of Vidyasagar\cite{Vidyasagar}, i.e., Theorem \ref{thm:IMP} is a generalization of the classical internal model principle for Bezout domains.
\end{remark}

A consequence of Theorems \ref{thm:IMP} and \ref{thm:Simplification} is that a robustly regulating controller maintains its functionality whenever the signal generator is perturbed so that no additional instabilities are generated. E.g., one can multiply the signal generator by a stable matrix. This can be understood as a limited type of robustness to perturbations of the signal generator and is formulated in the next corollary. One should also note that having a non-minimal internal model in the sense of the above discussion allows some additional unstable dynamics in the perturbed signal generator.

\begin{corollary}\label{cor:PertubGen}
Let a controller $\Cont$ robustly regulate $\Plant$ with the signal generator $\Gen_r=(\gen_{ij})$. Then $\Cont$ robustly regulates $\Plant$ with the signal generator $\Phi_r=(\phi_{ij})$ if $\langle \phi_{ij} | 1\leq i\leq n, 1\leq j\leq q\rangle\subseteq \langle \gen_{ij} | 1\leq i\leq n, 1\leq j\leq q\rangle$.
\end{corollary}

\begin{example}
Assume that $\Cont$ solves the robust regulation problem with the signal generator $\Gen_r$ in Example \ref{exa:Simplification} and let $\gen$ be the generating element found in the example. The controller is then robustly regulating with the signal generator $\gen I$ and Corollary \ref{cor:PertubGen} implies that the controller remains robustly regulating for a perturbed signal generator $\Phi_r$ whenever $\Phi_r=\gen \Phi_0$ for some stable $\Phi_0$. In such signal generators some of the unstable poles may cancel out, but no new unstable poles can appear.
\end{example}

%%%%%%%%%%%%%%%%%%%%%%%%%%%%%%%%%%%%%%%%%%%%%%%%%%%%%%%%%%%%%
%% ==> The internal model principle:
\section{Solvability Conditions and Controller Design}\label{sec:Solv}

%The main result of this section is given in the following theorem. It gives a necessary and sufficient solvability condition for the robust control problem. A procedure for designing a controller solving the problem is given based on it.

Solvability of the robust regulation problem is studied in this section. The main theoretical results give solvability conditions with varying assumptions on the plant and the signal generator. The theoretical results lead to specific controller design procedures and a parametrization of robustly regulating controllers.
% As discussed in Section \ref{sec:Simplification}, one can simplify the design procedure by first finding a small number of elements capturing the type of instability the generator can generate. Because of this, every design procedure proposed here starts by simplifying step.
%finding a set $\{f_1,\ldots,f_k\}$, preferrably with small $k$, such that $\langle \gen_{ij}\, |\, i\in\List{\ind},\, j\in\List{q}\rangle\subseteq\langle f_1,\ldots,f_k\rangle$. 
%This has two advantages. First, it simplifies the process as it reduces the number of separate matrix equations to be solved. Secondly, in some design procedures one may avoid having an oversized internal model this way.

The solvability conditions of this section are given for the robust regulation problem and the design procedures can lead to non-causal controllers. Theorem \ref{thm:ExistenceOfCausalControllers} implies that the found controllers are causal whenever the given plant is strictly causal. Because of Remark \ref{rem:RRPvsCRRP}, all controllers found by the design procedures solve the causal robust regulation problem, but the solvability conditions of Theorems \ref{thm:SolvabilityStable}, \ref{thm:Solvability}, \ref{thm:SolvabilityCoprime}, and \ref{thm:SolvabilitySimpleGen} may not be necessary. %For strictly causal controllers the two problems are equivalent by Corollary \ref{cor:IMPStrictlyCausal} and consequently the solvability conditions apply for causal robust regulation problem as they are. 
These important observations are repeated in the following remark for future use. Additional observations specific to the different cases studied are presented in the subsections.

\begin{remark}\label{rem:CausalityInDesign}
The following observations apply to the results of this section:
\begin{enumerate}
\item If the given plant $\Plant$ is strictly causal, then the solvability conditions guarantee existence of a causal robustly regulating controller and the controllers constructed by the design procedures are causal.

\item A controller $\Cont$ found by any of the proposed design procedures solves the causal robust regulation problem. However, the solvability conditions of Theorems \ref{thm:SolvabilityStable}, \ref{thm:Solvability}, \ref{thm:SolvabilityCoprime}, and \ref{thm:SolvabilitySimpleGen} may not be necessary for the causal robust regulation problem and further study is needed.
\end{enumerate}
\end{remark}

\subsection{Case I: Stable Plants}

A solvability condition for stable plants is given first. It is inspired by the solvability condition suitable for rational transfer matrices given in Theorem 7.5.2 of Vidyasagar\cite{Vidyasagar}. Here the plant is assumed to be stable, but the signal generator does not need to possess a coprime factorization. The result shows that the robust regulation problem is solvable if and only if the plant does not block the unstable dynamics of the reference signals.
% Later in Theorem \ref{thm:SolvabilityNecessary} it is shown that this is a necessary solvability condition for an unstable system as well, but not sufficient 
%to guarantee that we can include an internal model into a stabilizing controller 
%in general.

% The first one actually implies that exsitence of a coprime factorization implies the existence of a stabilizing controller which we need later. This result was given for example in \cite{Vidyasagaretal1982, Quadrat2006b}. However, in \cite{Quadrat2006b} the result of the lemma was assumed and in \cite{Vidyasagaretal1982} an additional assumption that $\stable$ is a normed algebra was made.

\begin{theorem}\label{thm:SolvabilityStable}
Assume that $\Plant\in\matrices{\stable}$. The robust regulation problem is solvable if and only if for all $i\in\List{\outd}$ and $j\in\List{q}$ the equation
\begin{eqnarray}\label{eq:SolvabilityStable}
I= \gen_{ij}\inv A_{ij}-\Plant B_{ij},
\end{eqnarray}
is solvable by some $A_{ij}\in \stable^{\outd\times \outd}$ and $B_{ij}\in \stable^{\ind\times \outd}$ whenever $\gen_{ij}$ is non-zero.
\end{theorem}
\begin{proof}
In order to show necessity, assume that $\Cont$ is a robustly regulating controller. Then $A_{ij}:=\gen_{ij}(I-\Plant\Cont)\inv\in\matrices{\stable}$ by Theorem \ref{thm:GenTFDiv} and $B_{ij}:=\Cont(I-\Plant\Cont)\inv\in\matrices{\stable}$ since $\Cont$ is stabilizing. It follows that
\begin{eqnarray*}
I=\gen_{ij}\inv \gen_{ij}(I-\Plant\Cont)\inv-\Plant\Cont(I-\Plant\Cont)\inv=\gen_{ij}\inv A_{ij}-\Plant B_{ij}.
\end{eqnarray*}
It remains to show sufficiency. For simplicity, reorder the elements $\gen_{ij}$ of the signal generator to $f_l$ where $l\in\List{k}$ with $k=\outd q$ and write \eqref{eq:SolvabilityStable} in the form
\begin{eqnarray}
\label{eq:SolvabilityStableSimplified}
I= f_{l}\inv A_{l}-\Plant B_{l}.
\end{eqnarray} 
By Lemma \ref{lem:exInvA}, without loss of generality, one may assume that $A_l$ is invertible over $\fractions{\stable}$. Since $\Plant B_l$ is stable as a product of two stable matrices, \eqref{eq:SolvabilityStableSimplified} reveals that $f_l\inv A_l\in\matrices{\stable}$ for every $l\in\List{k}$. Thus, for $r\in\List{k}$, $\mathcal{A}_r:=\prod_{l=0}^{r-1} f_{r-l}\inv A_{r-l}$ and $\mathcal{B}_r:=B_1+\sum_{h=2}^{r} B_h \mathcal{A}_{h-1}$ are stable matrices. In addition, one can show by induction that
\begin{eqnarray}\label{eq:StabStab}
I=\mathcal{A}_r-\Plant\mathcal{B}_r
\end{eqnarray}
%since \eqref{eq:SolvabilityStableSimplified}, the above equation holds for $r=1$. Next assume that \eqref{eq:StabStab} holds for $r<k$. By \eqref{eq:SolvabilityStableSimplified},
%\begin{eqnarray*}
%I & =\mathcal{A}_r-\Plant\mathcal{B}_r\\
%& = (f_{r+1}\inv A_{r+1}-\Plant B_{r+1})\mathcal{A}_r-\Plant\mathcal{B}_r\\
%& = f_{r+1}\inv A_{r+1}\mathcal{A}_{r}-\Plant B_{r+1}\mathcal{A}_r-\Plant \mathcal{B}_r\\
%& = \mathcal{A}_{r+1}-\Plant\mathcal{B}_{r+1}.
%\end{eqnarray*}
%Equation \eqref{eq:StabStab} holds for $r=k$ by induction. 
for all $1\leq r\leq k$. Set $\mathcal{A}=\mathcal{A}_k$ and $\mathcal{B}=\mathcal{B}_k$. Since $\mathcal{A},\mathcal{A}\Plant,\mathcal{B},\mathcal{B}\Plant\in\matrices{\stable}$ and \eqref{eq:StabStab} holds for $r=k$, Proposition 6 of Quadrat\cite{Quadrat2006b} implies that $\Cont=\mathcal{B}\mathcal{A}\inv$ is stabilizing for $\Plant$.

It remains to show that $\Cont$ contains an internal model. Since $\mathcal{A} = (I-\Plant\Cont)\inv$, the equation \eqref{eq:StabStab} is of the form 
$
%\begin{eqnarray*}
I  %= \mathcal{A}-\Plant\Cont\mathcal{A} 
=\mathcal{A}-\mathcal{A}\Plant\Cont
%\end{eqnarray*}
$. 
Multiplying it by $f_l$ and observing that $f_l\mathcal{A}\in\matrices{\stable}$ shows that \eqref{eq:IMP} holds and the claim follows by Theorem \ref{thm:IMP}.
\end{proof}
The above theorem and its proof implies the following controller design procedure for a stable plant. The idea is to construct the unstable dynamics generated by the signal generator into the controller element by element.

\begin{design}\label{des:DesignStable}
Define the controller 
%\begin{eqnarray}\label{eq:ControllerStable}
$
\Cont=\mathcal{B}\mathcal{A}\inv
%\end{eqnarray}
$ 
where the stable parameters $\mathcal{B}$ and $\mathcal{A}$ are chosen in the following way:
%\textbf{Part I: Simplification of the internal model}
\begin{enumerate}[\bf Step 1:]
\item Find a set of non-zero elements $f_1,\ldots,f_k\in\fractions{\stable}$ such that
%\begin{eqnarray*}
$
\langle f_1,\ldots,f_k\rangle=\langle \gen_{ij} | 1\leq i\leq n, 1\leq j\leq q\rangle.
$
%\end{eqnarray*}
%\setcounter{listNo}{\value{enumi}}
%\end{enumerate}
%\textbf{Part II: Construction of the internal model}
%\begin{enumerate}[\bf Step 1:]
%\setcounter{enumi}{\value{listNo}}
\item Set $\mathcal{A}_0=I$, $\mathcal{B}_0=0$, and $l=1$. 
\item If possible, find $A_l, B_l\in\matrices{\stable}$ such that 
%\begin{eqnarray}\label{eq:SolvabilityStable2}
$
I= f_l\inv A_{l}-\Plant B_{l}
$ 
%\end{eqnarray}
and $\det(A_l)\neq 0$. Define $\mathcal{A}_l:=f_l\inv A_{l}\mathcal{A}_{l-1}$ and $\mathcal{B}_l:=\mathcal{B}_{l-1}+ B_l \mathcal{A}_{l-1}$.
If such matrices cannot be found, end the procedure since the robust regulation problem is not solvable.

\item If $l=k$, set $\mathcal{A}=\mathcal{A}_l$ and $\mathcal{B}=\mathcal{B}_l$ and end the procedure. Otherwise, set $l=l+1$ and return to Step 3.
\end{enumerate}
\end{design}

\begin{remark}\label{rem:SimplificationInDesign}
Obviously one can always use the elements $\gen_{ij}$ of the signal generator in Step 1. However, the significance of Step 1 is that one can get rid of the unstable dynamics shared by one or more elements of the signal generator. If this is not done and the elements $\gen_{ij}$ share some unstable dynamics or if the generating elements $f_l$ are not chosen with care, the procedure would result into an oversized internal model, since the same unstable dynamics are constructed into the controller repeatedly. This increases the size of a state-space realization of the controller.

In order to illustrate this, consider Example \ref{exa:Simplification}. The unstable pole of order one at $s=0$ appears in two elements of the signal generator. This pole would appear as a second order pole in the controller without the first step. The simplified internal model, i.e., the element $\frac{(s+1)^2}{(\e^{-s}-1)(s^2+\pi^2)}$, only has a first order pole at $s=0$. Consequently, only a first order pole at $s=0$ is required in the controller.

It may not be easy to find the generating elements $f_l$ %of $\langle f_1,\ldots,f_k\rangle=\langle \gen_{ij} | 1\leq i\leq n, 1\leq j\leq q\rangle$ 
so that the same unstable dynamics are not repeated. At least they should form a minimal generating set and finding one is not a trivial task. Having a minimal generating set may not be enough as illustrated by Example \ref{exa:Simplification} since the two non-zero elements of the signal generator form such a set for the corresponding fractional ideal. The situation is clear if one finds a single generating elements $f_1$ with a weakly coprime factorization as explained in Remark \ref{rem:SimpleInternalModel}.
\end{remark}

\begin{remark}\label{rem:SolvabilityInDesign}
By Theorem \ref{thm:Simplification}, the above design procedure can be completed whenever the robust regulation problem is solvable since the fractional ideals in the first step are equal. Then the solvability conditions are checked in Step 3 and the failure of this step means that the robust regulation problem is not solvable. One can choose $\langle f_1,\ldots,f_k\rangle$ having $\langle \gen_{ij} | 1\leq i\leq n, 1\leq j\leq q\rangle$ as its proper subset in Step 1, but in that case the failure of Step 3 would not necessarily imply that the robust regulation problem is not solvable. Another downside of having inequality is that the resulting internal model is oversized if the procedure is successful.
\end{remark}

\begin{remark}
One may add the internal model of new unstable dynamics into an existing controller using Design procedure 1. First observe that the existing controller gives the solution to \eqref{eq:SolvabilityStable} with the old dynamics by the proof of Theorem \ref{thm:SolvabilityStable}. The remaining task is to repeat Step 3 of the procedure with the unstable dynamics to be added into the controller. 
\end{remark}
%\begin{remark}
%Because of the generality of the approach, no specific algorithm for carrying out Part I can be given. However, ring specific algorithms may exist. E.g., for $\RHinf$ such an algorithm is well-known \cite[Chapter 6]{Kailath}.
%\end{remark}

%Consider a matrix $M\in\matrices{\fractions{\stable}}$ with a right factorization $\Plant=ND^{-1}$. If $D$ is $\primeideal$-nonsingular, then $\Plant=ND^{-1}=N\cdot\mathrm{adj}(D)\cdot\frac{1}{\det(D)}$ where $N\mathrm{adj}(D)\in\matrices{\stable}$ and $\det(D)\in\stable\setminus\primeideal$. This means that $\Plant$ is causal.

\begin{remark}\label{rem:Design1Causality}
A matrix is causal if it has a factorization whose denominator is $\primeideal$-nonsingular.\cite{MoriAbe2001} The product of $\primeideal$-nonsingular matrices is $\primeideal$-nonsingular, so the controller found using Design procedure 1 is causal if the matrices $A_l$ can be chosen so that $f_l\inv A_l$ are $\primeideal$-nonsingular.
\end{remark}

\subsection{Case II: General Plants}

The following is the main result of this section, and it gives a necessary and sufficient condition for the solvability of the robust regulation problem using no coprime  factorizations. It results to a straightforward controller design procedure, where the plant is first stabilized and then a stabilizing controller containing the internal model is determined by exploiting the parametrization of all stabilizing controllers.

\begin{theorem}\label{thm:Solvability}
Let $\Cont_s$ stabilize $\Plant$ and write $\gen_{ij}=\frac{n_{ij}}{d_{ij}}$. Denote $U=(I-\Plant\Cont_s)\inv$, $V=\Cont_s(I-\Plant\Cont_s)\inv$, $\widetilde{L}=\begin{bmatrix}
-V & (I-\Cont_s\Plant)\inv\end{bmatrix}$, and $L=\begin{bmatrix}
U^T & V^T
\end{bmatrix}^T$.
The robust regulation problem is solvable if and only if the system of equations
\begin{eqnarray}\label{eq:Solvability}
\left(d_{ij}\widetilde{A}_{ij}-\begin{bmatrix}
n_{ij}I & 0
\end{bmatrix}
\right)\left(I+\begin{bmatrix}
P\widetilde{L}\\\widetilde{L}
\end{bmatrix}W \right)L=0
\end{eqnarray}
%\begin{eqnarray}\label{eq:Solvability}
%(n_{ij}I-d_{ij} A_{ij})(U+\Plant\widetilde{L}WL)=B_{ij}(V+\widetilde{L}WL)
%\end{eqnarray}
where $i\in\List{\outd}$ and $j\in\List{q}$, is solvable by $\widetilde{A}_{ij}\in \stable^{\outd\times (\outd+\ind)}$ and $W\in\stable^{(\outd+\ind)\times (\outd+\ind)}$ such that $\det(U+P\widetilde{L}WL)\neq 0$.
\end{theorem}
\begin{proof}
Writing $\widetilde{A}_{ij}=\begin{bmatrix}
A_{ij} & B_{ij}\end{bmatrix}$ where $A_{ij}\in \stable^{\outd\times \outd}$ and $B_{ij}\in \stable^{\outd\times \ind}$, the equation \eqref{eq:Solvability} can be written in the form
\begin{eqnarray*}
\begin{bmatrix}
d_{ij} A_{ij}-n_{ij}I & d_{ij} B_{ij}
\end{bmatrix}\begin{bmatrix}
U+P\widetilde{L}W L\\V+\widetilde{L}W L
\end{bmatrix}=0,
\end{eqnarray*}
which is equivalent to
\begin{eqnarray}\label{eq:SolvabilityUnstable}
\gen_{ij} I=A_{ij}+B_{ij}(V+\widetilde{L}WL)(U+\Plant\widetilde{L}WL)\inv
\end{eqnarray}
if $\det(U+P\widetilde{L}WL)\neq 0$. The result follows by observing that any stabilizing controller $\Cont$ can be expressed in the form $\Cont=(V+\widetilde{L}WL)(U+\Plant\widetilde{L}WL)\inv$ by Theorem \ref{thm:Stability} and that \eqref{eq:SolvabilityUnstable} is equivalent of it being robustly regulating by Theorem \ref{thm:IMP}.
\end{proof}

If $\langle\gen_{ij} | 1\leq i\leq n, 1\leq j\leq q\rangle\subseteq \langle f_1,\ldots,f_k\rangle$, then Theorem \ref{thm:Simplification} implies that solving \eqref{eq:Solvability} for every $f_i$ instead of $\gen_{ij}$ results to a robustly regulating controller. % just like in the above proof. 
Writing the system of equation \eqref{eq:Solvability} in the matrix form \eqref{eq:SolvabilityMatrix} leads to the following design procedure.

\begin{design}\label{des:DesignGeneral}
Define the controller
\begin{eqnarray}\label{eq:RobCont}
\Cont=(V+\widetilde{L}WL)(U+\Plant\widetilde{L}WL)\inv
\end{eqnarray}
where the parameters $U$, $L$, $\widetilde{L}$, and $W$ are chosen by the following procedure:
%\textbf{Part I: Simplification of the internal model}
\begin{enumerate}[\bf Step 1:]
\item Find a set of non-zero elements $f_1,\ldots,f_k\in\fractions{\stable}$ such that
%\begin{eqnarray*}
$
\langle f_1,\ldots,f_k\rangle=\langle \gen_{ij} | 1\leq i\leq n, 1\leq j\leq q\rangle.
$
%\end{eqnarray*}
%\setcounter{listNo}{\value{enumi}}
%\end{enumerate}
%\textbf{Part II: Determine a Stabilizing Controller Possessing an Internal Model}
%\begin{enumerate}[\bf Step 1:]
%\setcounter{enumi}{\value{listNo}}
\item Find a stabilizing controller $\Cont_s$ for $\Plant$ 
and let $U$, $V$, $\widetilde{L}$, and $L$ be as in Theorem \ref{thm:Solvability}.
\item Write $f_i=n_i/d_i$ and find the parameter $W$ by solving the matrix equation
\begin{eqnarray}\label{eq:SolvabilityMatrix}
\left(\begin{bmatrix}
d_1 I & 0 & \dots & 0\\
0 & d_2 I & \dots & 0\\
\vdots & \vdots & \ddots & \vdots\\
0 & 0 & \dots & d_k I
\end{bmatrix}\begin{bmatrix}
\widetilde{A}_1\\
\widetilde{A}_2\\
\vdots\\
\widetilde{A}_{k}
\end{bmatrix}+
\begin{bmatrix}
n_1 I & 0\\
n_2 I & 0\\
\vdots & \vdots\\
n_{k} I & 0
\end{bmatrix}
\right)\left(I+\begin{bmatrix}
\Plant\widetilde{L}\\\widetilde{L}
\end{bmatrix}
W
\right)L=0.
\end{eqnarray}
\end{enumerate}
\end{design}

Remark %s \ref{rem:SimplificationInDesign} and 
\ref{rem:SolvabilityInDesign} concerning the first step applies to the above design procedure. In addition, the size of the matrix equation \eqref{eq:SolvabilityMatrix} is reduced if the number of elements $f_l$ in the first step is small.

\begin{remark}
The point of using equations \eqref{eq:Solvability} and \eqref{eq:SolvabilityMatrix} instead of \eqref{eq:SolvabilityUnstable} is that they involve only stable matrices. 
%since the $\Cont_s$ is stabilizing and consequently $\Plant\widetilde{L}\in\matrices{\stable}$. 
Then the proposed design procedure requires solving a matrix equation of the form $(M_1+M_2 X)(M_3+M_4 Y)M_5=0$ over the ring $\stable$ where the matrices $X,Y$ are to be solved.
\end{remark}

The following necessary solvability condition shows that the robust regulation problem is solvable only if the plant does not block the unstable dynamics produced by the signal generator. The condition is not sufficient since it does not address the stabilizability of $\Plant$ in anyway.

\begin{theorem}\label{thm:SolvabilityNecessary}
Assume that the robust regulation problem is solvable. Then for all $i\in\List{\outd}$ and $j\in\List{q}$ the equation 
%\begin{eqnarray}\label{eq:SolvabilityNecessary}
$
\gen_{ij}I= A_{ij}-\gen_{ij}\Plant B_{ij}
$ 
%\end{eqnarray}
is solvable by some $A_{ij}\in \stable^{\outd\times \outd}$ and $B_{ij}\in \stable^{\ind\times \outd}$.
\end{theorem}
\begin{proof}
If $\Cont$ is a robustly regulating controller, then $\gen_{ij}I=\gen_{ij}(I-\Plant\Cont)\inv-\gen_{ij}\Plant\Cont(I-\Plant\Cont)\inv$ where
$\gen_{ij}(I-\Plant\Cont)\inv\in\matrices{\stable}$ and $\Cont(I-\Plant\Cont)\inv\in\matrices{\stable}$.
\end{proof}

The next theorem generalizes the solvability condition of stable plants given in Theorem \ref{thm:SolvabilityStable} to general plants. This condition is only sufficient.
%The solvability condition of Theorem \ref{thm:SolvabilitySufficient} is a generalization of the condition for stable plants given in Theorem \ref{thm:SolvabilityStable} to general plants. The new condition is only sufficient.
%\cite[Theorem 7.5.2]{Vidyasagar} for rational transfer matrices. 
The restatement of this result for systems with a coprime factorization given later in Theorem \ref{thm:SolvabilityCoprime} is both necessary and sufficient. Roughly speaking, the idea is that finding a numerator of the plant that does not block the unstable dynamics produced by the signal generator guarantees existence of a robustly regulating controller.

\begin{theorem}\label{thm:SolvabilitySufficient}
The robust regulation problem is solvable if there exists a stabilizing controller $\Cont_s$ of $\Plant$ such that for all $i\in\List{\outd}$ and $j\in\List{q}$, the equation
\begin{eqnarray}\label{eq:SolvabilitySufficient}
\gen_{ij} I= A_{ij}-\gen_{ij}\Plant(I-\Cont_s\Plant)\inv B_{ij}
\end{eqnarray}
is solvable by some $A_{ij}\in \stable^{\outd\times \outd}$ and $B_{ij}\in \stable^{\ind\times \outd}$.
\end{theorem}
\begin{proof}
By Theorem \ref{thm:SolvabilityStable}, the equation \eqref{eq:SolvabilitySufficient} implies that there exists a controller $\Cont_r$ that robustly regulates $\Plant_0=\Plant(I-\Cont_s\Plant)\inv$. The claim follows if one can show that $\Cont=\Cont_s+\Cont_r$ is robustly regulating for $\Plant$. Lemma \ref{lem:TwoStageStab} implies that $\Cont$ is stabilizing. Now
\begin{eqnarray*}
\gen_{ij}(I-\Plant\Cont)\inv & =&\gen_{ij}(I-(I-\Plant\Cont_s)\inv\Plant\Cont_r)\inv(I-\Plant\Cont_s)\inv\\
& = & \gen_{ij}(I-\Plant_0\Cont_r)\inv(I-\Plant\Cont_s)\inv\in\matrices{\stable}
\end{eqnarray*}
since $\Cont_s$ stabilizes $\Plant$ and $\Cont_r$ robustly regulates $\Plant_0$, i.e., $\gen_{ij}(I-\Plant_0\Cont_r)\inv\in\matrices{\stable}$. The matrix in the above equation remains stable if it is multiplied by $\Plant$ from the right since $(I-\Plant\Cont_s)\inv P\in\matrices{\stable}$. Theorem \ref{thm:GenTFDiv} implies that $\Cont$ solves the robust regulation problem.
\end{proof}

The above theorem shows that one way of finding a robust controller is the two-stage controller design where one first stabilizes the plant and then constructs a robustly regulating controller for the stabilized plant. One way of completing the second part is given in Design procedure \ref{des:DesignStable}. Combining the two controllers leads to a robustly regulating controller for the given plant. This idea is summarized by the following controller design procedure.

\begin{design}\label{des:DesignSufficient}
Define the controller 
$
%\begin{eqnarray}\label{eq:RobCont2}
\Cont=\Cont_s+\Cont_r
%\end{eqnarray}
$ 
where $\Cont_s$ and $\Cont_r$ are chosen in the following way:
%\textbf{Part I: Simplification of the internal model}
\begin{enumerate}[\bf Step 1:]
%\item Find a set of non-zero elements $f_1,\ldots,f_k\in\fractions{\stable}$ such that 
%\begin{eqnarray*}
%$
%\langle f_1,\ldots,f_k\rangle\supseteq \langle \gen_{ij} | 1\leq i\leq n, 1\leq j\leq q\rangle,
%$ 
%\end{eqnarray*}
%preferrably with equality instead of inclusion and with small $k$.
%\setcounter{listNo}{\value{enumi}}
%\end{enumerate}
%\textbf{Part II: Stabilization}
%\begin{enumerate}[\bf Step 1:]
%\setcounter{enumi}{\value{listNo}}
\item Find a stabilizing controller $\Cont_s$ for $\Plant$.
%\setcounter{listNo}{\value{enumi}}
%\end{enumerate}
%\textbf{Part III: Construction of the internal model}
%\begin{enumerate}[\bf Step 1:]
%\setcounter{enumi}{\value{listNo}}
\item Find a robustly regulating controller $\Cont_r$ for $P_0=\Plant(I-\Cont_s\Plant)\inv$.
%For all $\i\in\List{k}$, find matrices $A_i\in \stable^{\outd\times\outd}$ and $B_i\in\stable^{\ind\times \outd}$ such that
%\begin{eqnarray}\label{eq:RobContDesignSR2}
%I=f\inv_i A_i-\Plant(I-\Plant\Cont_s)\inv B_i
%\end{eqnarray}
%and $A_i$ is invertible.
%\item Define $C_r=\mathcal{B}\mathcal{A}\inv$ where
%\begin{eqnarray}\label{eq:RobContDesignRR2}
%\mathcal{A}=\prod_{i=0}^{k-1} f_{k-i}\inv A_{k-i}\quad\text{and}\quad \mathcal{B}=B_1+\sum_{j=2}^{k} \left(B_j \prod_{i=1}^{j-1} f_{j-i}\inv A_{j-i}\right)
%\end{eqnarray}
\end{enumerate}
\end{design}

\begin{remark}
The above design procedure may fail %even if the robust regulation problem is solvable and one has equality of the fractional ideals in the first step. This is 
because the stabilizing controller in the first step may already contain a partial internal model causing the stable plant $\Plant_0$ to block some unstable dynamics of the reference signals and, consequently, Step 2 to fail. In such a case one may be able to complete the last step by ignoring the unstable dynamics already appearing in the stabilizing controller as is done in Example \ref{exa:Delay}.
\end{remark}

\begin{remark}\label{rem:Design3Causality}
The controller found by using Design procedure 3 is causal if the stabilizing controller $\Cont_s$ and the robustly regulating controller $\Cont_r$ are both causal. The observations made in Remarks \ref{rem:CausalityInDesign} and \ref{rem:Design1Causality} apply to finding the causal robust controller.% In particular, if $\Cont_s$ is such a causal stabilizing controller that $\Plant_0$ is strictly causal, then $\Cont_r$ is necessarily causal. 
\end{remark}

\subsection{Case III: Plants with Right Coprime Factorizations}

The following theorem generalizes the solvability condition of Theroem \ref{thm:SolvabilityStable} to plants with a coprime factorization. The theorem generalizes the solvability condition of Laakkonen \cite{Laakkonen2016} by allowing a general signal generator.

\begin{theorem}\label{thm:SolvabilityCoprime}
Provided that a stabilizable plant $\Plant$ has a right coprime factorization $\Plant=ND^{-1}$, the robust regulation problem is solvable if and only if for every non-zero $\gen_{ij}$ where $i\in\List{\outd}$ and $j\in\List{q}$, the equation
\begin{eqnarray}\label{eq:SolvabilityCoprime}
I=\gen_{ij}\inv A_{ij}-N B_{ij}
\end{eqnarray}
is solvable by some $A_{ij}\in \stable^{\outd\times \outd}$ and $B_{ij}\in \stable^{\ind\times \outd}$.
\end{theorem}
\begin{proof}
Since $\Plant=ND\inv$ is a right coprime factorization, Lemma 8.3.2 of Vidyasagar\cite{Vidyasagar} and its proof imply that any stabilizing controller can be written in the form $\Cont_s=X\inv Y$ where $X,Y\in\matrices{\stable}$ are such that 
\begin{eqnarray}\label{eq:Coprime}
I=XD-YN
\end{eqnarray}
and $\det(X)\neq 0$. If $\Cont_s=X\inv Y$ satisfying \eqref{eq:Coprime} solves the robust regulation problem, then a direct calculation shows that
$$
I=(I-\Plant\Cont_s)\inv-\Plant\Cont_s(I-\Plant\Cont_s)\inv=\gen_{ij}\inv\gen_{ij}(I-\Plant\Cont_s)\inv-NY,
$$
where $\gen_{ij}(I-\Plant\Cont_s)\inv\in\matrices{\stable}$. This implies necessity.

In order to show sufficiency, choose an arbitrary stabilizing controller $\Cont_s$ and let $\Cont_s=X\inv Y$ be its coprime factorization satisfying \eqref{eq:Coprime}. By Theorem \ref{thm:SolvabilityStable}, the equation \eqref{eq:SolvabilityCoprime} implies that there exists a robustly regulating controller $\Cont_r$ for $N$. Consider the controller 
$
\Cont:=\Cont_s+X\inv \Cont_r
$. 
The chosen controller has the left coprime factorization $\Cont=(D_RX)\inv(D_RY+N_R)$ where $D_R=(I-\Cont_rN)\inv$ and $N_R=D_R\Cont_r$ are stable matrices since $\Cont_r$ stabilizes $N$. To verify this, observe that
\begin{eqnarray*}
D_RX D-(D_RY+N_R)N=D_R(XD-YN)-N_RN=D_R-D_R\Cont_r N=D_R(I-\Cont_r N)=I.
\end{eqnarray*}
This implies that $\Cont$ stabilizes $\Plant$. In addition, 
$
\gen_{ij}(I-\Plant\Cont)\inv=\gen_{ij}(I-N\Cont_r)\inv(I-\Plant\Cont_s)\inv\in\matrices{\stable}
$  
%is stable 
since $\Cont_r$ robustly regulates $N$ and $\Cont_s$ stabilizes $\Plant$. The above matrix remains stable if it is multiplied by $P$, so $\Cont$ is robustly regulating by Theorem \ref{thm:GenTFDiv}.
\end{proof}

The proof of the above theorem leads to the following design procedure. It is not specified how the robust controller for the stable transfer matrix in the last step is found since it can be done in various ways, e.g., by using Design procedure \ref{des:DesignStable} or the simple method applied in Example \ref{exa:HeatEq}.

\begin{design}\label{des:DesignCoprime}
Define the controller 
$
%\begin{eqnarray}\label{eq:RobContCoprime}
\Cont=\Cont_s+X\inv \Cont_r
%\end{eqnarray}
$ 
where $\Cont_s=X\inv Y$ and $\Cont_r$ are chosen in the following way:
%\textbf{Part I: Simplification of the internal model}
\begin{enumerate}[\bf Step 1:]
%\item Find a set of non-zero elements $f_1,\ldots,f_k\in\fractions{\stable}$ such that
%\begin{eqnarray*}
%$
%\langle f_1,\ldots,f_k\rangle=\langle \gen_{ij} | 1\leq i\leq n, 1\leq j\leq q\rangle.
%$
%\end{eqnarray*}
%\setcounter{listNo}{\value{enumi}}
%\end{enumerate}
%\textbf{Part II: Stabilization}
%\begin{enumerate}[\bf Step 1:]
%\setcounter{enumi}{\value{listNo}}
\item Find a right coprime factorization $\Plant=ND\inv$ of the plant and find a stabilizing controller $\Cont_s=X\inv Y$ by solving the equation $XD-YN=I$.
%\setcounter{listNo}{\value{enumi}}
%\end{enumerate}
%\textbf{Part III: Construction of the internal model}
%\begin{enumerate}[\bf Step 1:]
%\setcounter{enumi}{\value{listNo}}
\item Find a robustly regulating controller $\Cont_r$ for $N$.
\end{enumerate}
\end{design}

\begin{remark}
In the above design procedure, the solvability of the robust regulation problem is verified in the last step when trying to construct a robustly regulating controller for the numerator matrix. If such a controller exists, then the problem is solvable and otherwise not. 
%The basic idea of this design procedure is the same as in Design procedure 3 and Remark \ref{rem:Design3Causality} concerning causality of the controller applies here with obvious changes. %Remarks \ref{rem:SimplificationInDesign} and \ref{rem:SolvabilityInDesign} apply here as well.
\end{remark}

%\subsection{Case IV: Signal Generators generating a principal ideal with a coprime generating element}
\subsection{Case IV: Simple Signal Generators}

As the final result it is shown that the classical solvability condition of Theorem 7.5.2 in Vidyasagar \cite{Vidyasagar} and the related controller design method are applicable in the general framework provided that the internal model is captured by a single element with a weakly coprime factorization. In such a case, the parametrization of all stabilizing controllers can be applied to obtain a parametrization of all robustly regulating controllers.
%elements of the signal generator define a principal fractional ideal having a generator with a coprime factorization.

\begin{theorem}\label{thm:SolvabilitySimpleGen}
Assume that there exists $\gen\in\fractions{\stable}$ with a weakly coprime factorization $\gen=\frac{n}{d}$ such that it satisfies the equality $\langle\gen\rangle=\langle \gen_{ij}\, |\, 1\leq i\leq n, 1\leq j\leq q\rangle$. Then the robust regulation problem is solvable if and only if $\Plant_0=d\inv\Plant$ is stabilizable and there exist $A,B\in\matrices{\stable}$ such that
\begin{eqnarray}\label{eq:SolvabilitySimpleGen}
I = d A+\Plant B.
\end{eqnarray}
Provided that the robust regulation problem is solvable, a controller $\Cont$ solves it if and only if the contoller is of the form $\Cont=d\inv \Cont_0$ where $\Cont_0$ stabilizes $\Plant_0=d\inv\Plant$.
\end{theorem}
\begin{proof}
The necessity parts of the claims are shown first. If a robustly regulating controller exists, Theorems \ref{thm:IMPClassic} and \ref{thm:SolvabilityNecessary} show that \eqref{eq:SolvabilitySimpleGen} holds. It remains to show that if $\Cont$ solves the robust regulation problem then $\Cont_0=d\Cont$ stabilizes $\Plant_0$. Observe that
\begin{subequations}
\begin{eqnarray}
(I-\Plant_0\Cont_0)\inv &=&(I-\Plant\Cont)\inv\in\matrices{\stable},\label{eq:C0P0Stab1}\\
\Cont_0(I-\Plant_0\Cont_0)\inv &=&d\Cont(I-\Plant\Cont)\inv\in\matrices{\stable},\label{eq:C0P0Stab2}\\
(I-\Cont_0\Plant_0)\inv &=&(I-\Cont\Plant)\inv\in\matrices{\stable},\label{eq:C0P0Stab3}\\
(I-\Plant_0\Cont_0)\inv\Plant_0 &=&d\inv(I-\Plant\Cont)\inv\Plant\in\matrices{\stable},\label{eq:C0P0Stab4}
\end{eqnarray}
\end{subequations}
The stability in \eqref{eq:C0P0Stab4} follows since $\Cont$ is robustly disturbance rejecting for the signal generator $d\inv I$ by Theorems \ref{thm:GenTFDiv} and \ref{thm:IMPClassic}. The above equations imply that $\Cont_0$ stabilizes $\Plant_0$.

In order to show the sufficiency parts, it is shown that $\Cont=d\inv\Cont_0$ is stabilizing and robustly regulating for $\Plant$ if $\Cont_0$ stabilizes $\Plant_0$ and \eqref{eq:SolvabilitySimpleGen} holds. In order to show stability, observe that \eqref{eq:C0P0Stab1} and \eqref{eq:C0P0Stab3} hold since $\Cont_0$ stabilizes $\Plant_0$. In addition,
\begin{eqnarray*}
(I-\Plant\Cont)\inv\Plant & = & d(I-\Plant_0\Cont_0)\inv\Plant_0\in\matrices{\stable},
\end{eqnarray*}
and using \eqref{eq:SolvabilitySimpleGen} one observes that
\begin{eqnarray*}
\Cont(I-\Plant\Cont)\inv & = &d\inv\Cont_0(I-\Plant_0\Cont_0)\inv(dA+\Plant B)\\
 & = &\Cont_0(I-\Plant_0\Cont_0)\inv A+\Cont_0(I-\Plant_0\Cont_0)\inv\Plant_0 B\in\matrices{\stable}.
\end{eqnarray*}
This shows that $\Cont$ stabilizes $\Plant$. By the equation \eqref{eq:SolvabilitySimpleGen}, one has
\begin{eqnarray*}
d\inv I& = & d\inv (I-\Plant\Cont)\inv(dA+\Plant B)-d\inv (I-\Plant\Cont)\inv\Plant\Cont\\
& = & ((I-\Plant\Cont)\inv A+(I-\Plant_0\Cont_0)\inv\Plant_0 B)-(I-\Plant_0\Cont_0)\inv\Plant_0\,\Cont.
\end{eqnarray*}
Theorem \ref{thm:IMPClassic} implies that $\Cont$ solves the robust regulation problem.
\end{proof}

The signal generator is causal by assumption. It follows that the generating element $\gen$ in the above theorem must be causal and Lemma \ref{lem:WCoprimeFactCausality} implies that the denominator $d$ of its weakly coprime factorization is in $\stable\setminus\primeideal$. It follows that $\Plant_0$ is causal as a product of two causal elements. Theorem \ref{thm:ExistenceOfCausalControllers} implies that $\Plant_0$ has a causal stabilizing controller provided that it is stabilizable. This leads to the following corollary.

\begin{corollary}\label{cor:SimpleGenCausalCont}
The solvability condition presented in Theorem \ref{thm:SolvabilitySimpleGen} is necessary and sufficient for the existence of a causal robustly regulating controller.
\end{corollary}

The above theorem implies the following straightforward design method, where one first includes the internal model and then stabilizes the resulting system. The order in which the stabilization and the construction of an internal model is done is reversed when compared to the design procedures proposed above.

\begin{design}\label{des:SimpleGen}
Define the controller 
$
%\begin{eqnarray}\label{eq:RobContSimpleGen}
\Cont=d\inv\Cont_0
%\end{eqnarray}
$ 
where $\Cont_0$ and $d$ are chosen in the following way:
%\textbf{Part I: Simplification of the internal model}
\begin{enumerate}[\bf Step 1:]
\item Find $\gen\in\fractions{\stable}$ with a weakly coprime factorization $\gen=\frac{n}{d}$ such that
%\begin{eqnarray*}
$
\langle \gen\rangle=\langle \gen_{ij} | 1\leq i\leq n, 1\leq j\leq q\rangle
$.
%\end{eqnarray*}

%\setcounter{listNo}{\value{enumi}}
%\end{enumerate}
%\textbf{Part II: Checking solvability and inclusion of the internal model}
%\begin{enumerate}[\bf Step 1:]
%\setcounter{enumi}{\value{listNo}}
\item Check the solvability by solving the equation \eqref{eq:SolvabilitySimpleGen}. If this is not possible, end the procedure since the robust regulation problem is not solvable.
\item %Define $\Plant_0=d\inv \Plant$
%\setcounter{listNo}{\value{enumi}}
%\end{enumerate}
%\textbf{Part III: Stabilization}
%\begin{enumerate}[\bf Step 1:]
%\setcounter{enumi}{\value{listNo}}
Find a stabilizing controller $\Cont_0$ for $\Plant_0=d\inv P$.
\end{enumerate}
\end{design}

\begin{remark}
It is important to notice that the procedure without Step 2 can produce a controller that is not a robustly regulating controller. E.g., in the extreme case with $\Gen_r=d\inv I$ where $d$ is a non-unit element of $\stable$ and $\Plant=0$, the equation \eqref{eq:SolvabilitySimpleGen} obviously has no solution, but $\Plant_0=d\Plant=0$ is stable already. Of course, solving the equation is not necessary if the solvability can be verified in some other way since the solution is not used in the design procedure. One possibility is to skip Step 2 in the first place and use Theorem \ref{thm:GenTFDiv} or \ref{thm:IMPClassic} as a final step to verify that the resulting controller is indeed robustly regulating.
\end{remark}

\begin{remark}
Finding $\gen$ with a coprime factorization such that $\langle \gen_{ij} | 1\leq i\leq n, 1\leq j\leq q\rangle\subset \langle \gen\rangle$ is straightforward. E.g., one can write $\gen_{ij}=\frac{n_{ij}}{d_ij}$ and set $\gen=\frac{1}{d}$ where $d=\prod_{1\leq i\leq n,\,  1\leq j\leq q} d_{ij}$. Then the controller constructed in the above design procedure is robustly regulating provided that \eqref{eq:SolvabilitySimpleGen} is satisfied. Again the internal model may be oversized, but this way one can apply the procedure even if $\langle \gen_{ij} | 1\leq i\leq n, 1\leq j\leq q\rangle$ is not originally principal or the generator of the principal fractional ideal does not possess a weakly coprime factorization.
\end{remark}

Theorem \ref{thm:SolvabilitySimpleGen} implies that all robust controllers are found by finding all the stabilizing controllers of $\Plant_0$. Applying the parametrization for all stabilizing controllers given by Theorem \ref{thm:Stability} yields the following parametrization of robustly regulating controllers. For a strictly proper plant $\Plant$ all the controllers given by the parametrization are causal. If $\Plant$ is not strictly proper, then some of the controllers may be non-causal, but Corollary \ref{cor:SimpleGenCausalCont} guarantees that some of the controllers are causal.

\begin{corollary}\label{cor:ParamOfRRC}
Assume that $\langle\gen\rangle=\langle \gen_{ij}\, |\, 1\leq i\leq n, 1\leq j\leq q\rangle$ where $\gen\in\fractions{\stable}$ has a weakly coprime factorization $\gen=\frac{n}{d}$ and that $\Cont$ solves the robust regulation problem. Denote $L_0:=\begin{bmatrix}d\inv((I-\Plant\Cont)\inv)^T & (\Cont(I-\Plant\Cont)\inv)^T\end{bmatrix}^T$ and $\widetilde{L}_0:=
\begin{bmatrix}
-d(I-\Cont\Plant)\inv\Cont & (I-\Cont\Plant)\inv
\end{bmatrix}$. Then all controllers solving the robust regulation problem are given by the parametrization
\begin{subequations}
\label{eq:Parametrization}
\begin{eqnarray}
\Cont(W) & =&\left(\Cont(I-\Plant\Cont)\inv+\widetilde{L}_0WL_0\right)\left((I-\Plant\Cont)\inv+\Plant\widetilde{L}_0WL_0\right)\inv\\
& = & \left((I-\Cont\Plant)\inv+\widetilde{L}_0WL_0\Plant\right)\inv\left((I-\Cont\Plant)\inv\Cont+\widetilde{L}_0WL_0\right)
\end{eqnarray}
\end{subequations}
where $W\in\stable^{(n+m)\times(n+m)}$ is such an element that $\det((I-\Plant\Cont)\inv+\Plant\widetilde{L}_0WL_0)\neq 0$ and $\det((I-\Cont\Plant)\inv+\widetilde{L}_0WL_0\Plant)\neq 0$.
\end{corollary}

%%%%%%%%%%%%%%%%%%%%%%%%%%%%%%%%%%%%%%%%%%%%%%%%%%%%%%%%%%%%%
%% ==> Example:

\section{Examples}\label{sec:Example}

In the first example, Design procedure \ref{des:DesignGeneral} is applied to construct a robustly regulating controller for a delay system.
%The first example involves a transfer function of a delay system.
% The plant is chosen to be as simple as possible for ease of readability.
%The example illustrates the design procedure 2 in the ring $\pstable$ introduced in \cite{LaakkonenPohjolainen2015}. 
The reference signals have complicated unstable dynamics which restricts the possible choices of the ring $\stable$ of stable transfer functions. This underlines the importance of the general approach.

\begin{example}\label{exa:Delay}
Let the given plant and the signal generator be
% Distillation column 
\begin{eqnarray*}
\Plant=\begin{bmatrix}
\frac{\e^{-2s}}{s} & 0 \\
\frac{4\e^{-4s}}{1+2s} & \frac{2\e^{-2s}}{1+4s}
\end{bmatrix}
\quad\text{and}\quad
\Gen_r=\begin{bmatrix}
\frac{1}{\e^{-s}-1} & 0\\
0 & \frac{1}{s}+\frac{1}{s^2+\pi^2}
\end{bmatrix}.
\end{eqnarray*}
The aim is to find a robustly regulating controller by applying Design procedure \ref{des:DesignGeneral}.
%The time-domain interpretation of the signal generator is that the first output component should track an arbitrary $2\pi$-periodic signal and the second one the signal $1+\sin(\pi t)$.

Before proceeding a suitable ring of stable transfer functions and the ideal characterizing causality are fixed. The ring is chosen to be $\stable=\pstable$. The reason for this is that the signal generator has infinitely many poles on the imaginary axis,
which poses some restriction on the stability type achievable. E.g., it is not possible to solve the proposed robust regulation problem if $\stable$ is chosen to be $\Hinfty$.\cite{LaakkonenPohjolainen2015} The ideal $\primeideal$ is chosen to be 
\begin{equation*}
\primeideal=\left\{f\in\pstable\,\, \middle|\,\, \lim_{\rho\to 0}\sup_{\substack{\mathrm{Re}(s)\geq 0\\|s|>\rho}}|f(s)|=0\right\}.
\end{equation*}
This mean that (strictly) causal stable elements are exactly the elements that are (strictly) proper in the sense of the definition given by Curtain and Morris\cite{CurtainMorris2009}. It is easy to verify that $\primeideal$ is an ideal, but it is more complicated to show that it is prime. However, this is not needed since the plant possess right and left coprime factorizations. This follows since the stabilizing controller is rational and has coprime factorizations. Then strictly causal plants have only causal stabilizing controllers due to the results by Vidyasagar et al.\cite{Vidyasagaretal1982} mentioned in Remark \ref{rem:AlternativeCausality}. Finding the actual coprime factorizations is unnecessary.

% In order to solve the problem the boundedness requirement on the imaginary was relaxed resulting into a new ring 

The only unstable element of the plant can be written in the form $\left(\frac{s}{s+1}\right)\inv\frac{\e^{-2s}}{s+1}$ where $\frac{s}{s+1}\in\stable\setminus\primeideal$ and $\frac{\e^{-2s}}{s+1}\in\primeideal$, i.e., this element is strictly causal. The other elements are clearly in $\primeideal$, so the plant transfer matrix is strictly causal. This means that the controller found in the procedure will be causal. Similar arguments show that the signal generator is causal. It is not strictly causal since its first diagonal element does not vanish at infinity.

%Following the design procedure 2, 
\textbf{Step 1:} The signal generator is simplified first. The observation that $\frac{1}{\e^{-s}-1}\cdot \frac{\e^{-s}-1}{s}=\frac{1}{s}$ and $\frac{\e^{-s}-1}{s}\in\pstable$ leads to the equality
\begin{eqnarray}\label{eq:Exa1J}
\left\langle \frac{1}{\e^{-s}-1},0,0,\frac{1}{s}+\frac{1}{s^2+\pi^2}\right\rangle = \left\langle \frac{1}{\e^{-s}-1},\frac{1}{s^2+\pi^2}\right\rangle.
\end{eqnarray}
Denote $f_1=\frac{1}{\e^{-s}-1}$ and $f_2=\frac{1}{s^2+\pi^2}$. These elements have the fractional representations $f_j=\frac{n_j}{d_j}$ where $n_1=1$, $d_1=\e^{-s}-1$, $n_2=\frac{1}{(s+1)^2}$, and $d_2=\frac{s^2+\pi^2}{(s+1)^2}$ are elements of $\pstable$.

The calculations of Example \ref{exa:Simplification} show that the internal model is captured by the single element $\frac{(s+1)^2}{(\e^{-s}-1)(s^2+\pi^2)}$ also with $\stable=\pstable$. However, two elements are used in order to illustrate the matrix equation \eqref{eq:SolvabilityMatrix}. This does not lead to an oversized internal model since the chosen elements $f_1$ and $f_2$ do not have common unstable dynamics, i.e., unstable poles, unlike the original elements of the signal generator.

% As the second step, we
\textbf{Step 2:} The plant has a first order pole at $s=0$. This means that the stabilized closed loop is likely to have a zero at this pole appearing also in the signal generator. This would lead to a situation where no robust controller in the next step is available since the stabilized plant would block some unstable dynamics and that is not allowed by Theorem \ref{thm:SolvabilityNecessary}. The part of the internal model containing this pole is therefore built into the controller already at this step. This can be done using the PI-controller
\begin{eqnarray*}
\Cont_s=-\frac{1}{16}\left(4+\frac{1}{s}\right)I
\end{eqnarray*}
having the desired pole at the origin. This leads to
\begin{eqnarray*}
U&=(I-\Plant\Cont_s)\inv= 
\begin{bmatrix}
\frac{16s^2}{16s^2+(4s+1)\e^{-2s}} & 0\\
%\frac{64s^2(4s+1)\e^{-4s}}{(16s^2+(4s+1)\e^{-2s})(16s^2+2\e^{-2s})} & \frac{s}{16s^2+2\e^{-2s}}
\frac{32s^2(4s+1)\e^{-4s}}{(2s+1)(8s+\e^{-2s})(16s^2+(4s+1)\e^{-2s})} & \frac{8s}{8s+\e^{-2s}}
\end{bmatrix}, 
%\\
%V&=\Cont_s U =\begin{bmatrix}
%\frac{-s(4s+1)}{16s^2+(4s+1)\e^{-2s}} & 0\\
%\frac{-4s(4s+1)^2\e^{-4s}}{(16s^2+(4s+1)\e^{-2s})(16s^2+2\e^{-2s})} & \frac{-(4s+1)}{256s^2+32\e^{-2s}}
%\end{bmatrix}, \\
\end{eqnarray*}
$V=U\Cont_s=-\frac{4s+1}{16s}U$, $\widetilde{L}=\begin{bmatrix}-V & U\end{bmatrix}$, and $L=\begin{bmatrix}U\\ V\end{bmatrix}$. The parameters of the controller are chosen so that the plant is stabilized and the resulting matrices are as simple as possible. Naturally, this step can be carried out by using other types of controllers as well.

\textbf{Step 3:} It remains to choose $W$ and $\widetilde{A}_j$, $j\in\{1,2\}$, that solve \eqref{eq:SolvabilityMatrix}. The internal model should reproduce the poles of the signals generator. The pole at zero is already included in the stabilizing controller $\Cont_s$. Thus, $W$ should be chosen so that it captures the remaining poles of the signal generator. After choosing $W$ one should be able to choose $\widetilde{A}_j$ appropriately, which shows that the controller has an internal model of the signal generator.

To this end, a controller $\Cont_r$ containing all the unstable poles of the signal generator except the one at the origin is introduced. The controller is chosen to be of the form $\sum_k \frac{\varepsilon}{s-\omega_k}M_k$ where $\omega_k$ are the non-zero poles of the signal generator. It was shown by Laakkonen and Pohjolainen\cite{LaakkonenPohjolainen2015} that aligning $M_k$ appropriately with the stable plant to be regulated at each pole and choosing small enough gain $\varepsilon$ results to a robustly regulating controller. This in mind the controller
\begin{eqnarray*}
\Cont_r& =
\sum_{n\in\{-1,1\}}\frac{\varepsilon}{s+ \pi n\i}\left(P(\pi n\i)U(\pi n\i)\right)\inv
%\frac{\varepsilon}{(s- \pi \i)}\left(P(\pi \i)U(\pi \i)\right)\inv+\frac{\varepsilon}{(s+ \pi \i)}\left(P(-\pi \i)U(-\pi \i)\right)\inv
+\sum_{n\in\Z\setminus\{0\}}\frac{\varepsilon}{n^2(s- 2\pi n\i)}\left(P(2\pi n\i)U(2\pi n\i)\right)\inv\\
& = \varepsilon\begin{bmatrix}
\frac{4s+1-16\pi^2}{8(s^2+\pi^2)} & 0\\ \frac{4\pi^2(2s+1+8\pi^2)}{(1+4\pi^2)(s^2+\pi^2)} & \frac{12s+1-32\pi^2}{8(s^2+\pi^2)}
\end{bmatrix}+\varepsilon\sum_{n=1}^{\infty}\begin{bmatrix}
\frac{4s+1-64\pi^2n^2}{8n^2(s^2+(2\pi n)^2)} & 0\\ \frac{16\pi^2(2s+1+32\pi^2n^2)}{(s^2+(2\pi n)^2)(1+16\pi^2n^2)} & \frac{12s+1-128\pi^2n^2}{8n^2(s^2+(2\pi n)^2)}
\end{bmatrix}.
\end{eqnarray*}
that stabilizes $U\Plant$ is chosen. The series converges outside the poles since the numerator is of second order with respect to $n$ whereas the denominator is of fourth order. It follows that the matrix
\begin{eqnarray*}
W=\begin{bmatrix}
0_{2\times 2} & 0_{2\times 2}\\
\Cont_r(I-U\Plant\Cont_r)\inv & 0_{2\times 2}
\end{bmatrix}
\end{eqnarray*}
is stable over $\pstable$. Choose
\begin{eqnarray*}%\label{eq:Exa1Aj}
\widetilde{A}_j=f_j(I-U\Plant\Cont_r)\inv\begin{bmatrix}
U & -U\Plant
\end{bmatrix}
\end{eqnarray*}
where $j\in\{1,2\}$.
An analysis similar to that in Section 5.3 of Laakkonen and Pohjolainen \cite{LaakkonenPohjolainen2015} shows that $\widetilde{A}_2=f_2(I-U\Plant\Cont_r)\inv{\in}\matrices{\pstable}$ and $\frac{s}{(\e^{{-}s}-1)(s+1)}(I-U\Plant\Cont_r)\inv{\in}\matrices{\pstable}$. In addition, a direct calculation shows that $\frac{s+1}{s}\begin{bmatrix}U & -U\Plant\end{bmatrix}\in\matrices{\pstable}$. Thus, $\widetilde{A}_1$ is stable, i.e.,
%Write
\begin{eqnarray*}
\widetilde{A}_1=\frac{s}{(\e^{-s}-1)(s+1)}(I-U\Plant\Cont_r)\inv\cdot\frac{s+1}{s}\begin{bmatrix}
U & -U\Plant
\end{bmatrix}\in\matrices{\pstable}.
\end{eqnarray*} 
%$\widetilde{A}_j\in\matrices{\stable}$ for $j\in\{1,2\}$.

It remains to show that the chosen matrices $W$ and $\widetilde{A}_j$ satisfy the matrix equation \eqref{eq:SolvabilityMatrix} or equivalently the equations \eqref{eq:Solvability}. To that end, calculate
\begin{eqnarray}\label{eq:LWL}
\widetilde{L} WL &
= U \Cont_r(I-U\Plant \Cont_r)\inv U
\end{eqnarray} 
from which it follows that
\begin{subequations}\label{eq:Exa1CLWL1}
\begin{eqnarray}
\left(I+\begin{bmatrix}
\Plant\widetilde{L}\\
\widetilde{L}
\end{bmatrix} W\right)L &
= & \begin{bmatrix}
U\\
\Cont_s U
\end{bmatrix}
+
\begin{bmatrix}
\Plant U \Cont_r (I- U\Plant\Cont_r)\inv U\\
U\Cont_r (I-U\Plant\Cont_r)\inv U
\end{bmatrix}\\
& = &\begin{bmatrix}
I-\Plant U\Cont_r + U\Plant \Cont_r\\
\Cont_s(I-U\Plant \Cont_r)+U\Cont_r
\end{bmatrix}(I-U\Plant\Cont_r)\inv U\\
& = &\begin{bmatrix}
I\\
\Cont_s+\Cont_r
\end{bmatrix}(I-U\Plant\Cont_r)\inv U
\end{eqnarray} 
\end{subequations}
where the property $U\Plant =\Plant U$ was used. In addition, one can write
\begin{subequations}\label{eq:Exa1CLWL2}
\begin{eqnarray}
\left(d_j \widetilde{A}_j -\begin{bmatrix}n_j I & 0\end{bmatrix}\right) & = & n_j (I-U\Plant\Cont_r)\inv U
\begin{bmatrix}I-U\inv (I-U\Plant\Cont_r) & -\Plant\end{bmatrix} \\
& = & n_j (I-U\Plant\Cont_r)\inv U \Plant
\begin{bmatrix}\Cont_s +\Cont_r & -I\end{bmatrix}.
\end{eqnarray}
\end{subequations}
Showing that \eqref{eq:Solvability} holds can be done by substituting \eqref{eq:Exa1CLWL1} and \eqref{eq:Exa1CLWL2} into it.

Substitute $U$, $V$, and \eqref{eq:LWL} into \eqref{eq:RobCont} to obtain
\begin{eqnarray*}
\Cont & = &\left(\Cont_s U+U \Cont_r(I-U\Plant \Cont_r)\inv U\right)
\left(U+ U\Plant \Cont_r (I-U\Plant \Cont_r)\inv U\right)\inv\\
& = & \left(\Cont_s +U \Cont_r(I-U\Plant \Cont_r)\inv \right)
\left(I+ U\Plant \Cont_r (I-U\Plant \Cont_r)\inv \right)\inv\\
& = & \left(\Cont_s +U \Cont_r(I-U\Plant \Cont_r)\inv \right)
(I-U\Plant \Cont_r)\\
& = & \Cont_s+U(I-\Plant\Cont_s)\Cont_r = \Cont_s+\Cont_r.
\end{eqnarray*}
Thus, the constructed robustly regulating controller is
\begin{eqnarray*}
&\Cont  &=-\frac{1}{16}\left(4+\frac{1}{s}\right)I +\varepsilon\begin{bmatrix}
\frac{4s+1-16\pi^2}{8(s^2+\pi^2)} & 0\\ \frac{4\pi^2(2s+1+8\pi^2)}{(1+4\pi^2)(s^2+\pi^2)} & \frac{12s+1-32\pi^2}{8(s^2+\pi^2)}
\end{bmatrix}\\
&&\qquad\qquad\qquad\qquad+\varepsilon\sum_{n=1}^{\infty}\begin{bmatrix}
\frac{4s+1-64\pi^2n^2}{8n^2(s^2+(2\pi n)^2)} & 0\\ \frac{16\pi^2(2s+1+32\pi^2n^2)}{(s^2+(2\pi n)^2)(1+16\pi^2n^2)} & \frac{12s+1-128\pi^2n^2}{8n^2(s^2+(2\pi n)^2)}
\end{bmatrix}.
\end{eqnarray*}
Since the equation \eqref{eq:Simplification} holds with the choice $\stable=\pstable$, the fractional ideal \eqref{eq:Exa1J} is principal with the generator $\gen=\frac{(s+1)^2}{(\e^{-s}-1)(s^2+\pi^2)}$. Furthermore, $\gen$ has the coprime factorization $\frac{1}{d}$ where $d=\frac{(\e^{-s}-1)(s^2+\pi^2)}{(s+1)^2}\in\pstable$, so all robustly regulating controllers are given by the parametrization \eqref{eq:Parametrization}. All of them are causal since the plant is strictly causal. 
\end{example}

%The next example illustrates Design procedure 4 which has two main parts: first one stabilizes the plant and then finds a robust controller for a numerator matrix of the plant. A robust controller is a combination of the two controllers. The stabilization can be done using any available method. Here time domain methods are used. The key idea is that when 

In the following example, Design procedure 4 is applied to calculate a finite-dimensional robust controller for a heat equation. The example is particularly important since it demonstrates that the design procedure may be carried out using standard techniques and without calculating a closed form expression of the plant transfer matrix.

%It turns out that the only transfer function that needs to be explicitely calculated during the design procedure is the denominator of a coprime factorization of the stabilizing controller. Fortunately

%A finite-dimensional stabilizing controller is found using time-domain techniques. The robust controller for the numerator of the plant's coprime factorization can be found using the simple technique presented in \cite{HamalainenPohjolainen2000,RebarberWeiss2003}. It is sufficient to use approximations of the numerator matrix, so no closed form expression of the plant transfer function is required.  and there are standard techniques for finding a coprime factorization for it \cite{Kailath}. 

\begin{example}\label{exa:HeatEq}
Consider the heat equation
%\begin{subequations}
\begin{eqnarray*}
%\label{eq:heatState}
\frac{\partial z}{\partial t}(t,x)& = &\frac{\partial^2 z}{\partial x^2}(t,x),%\qquad
\qquad   z(0,x) =z_0(x)\\
%\label{eq:heatControl}
\frac{\partial z}{\partial x}(t,0) & =& -u_1(t),\qquad  
 \frac{\partial z}{\partial x}(t,1) =u_2(t)\\
%\label{eq:heatMeasurement}
y_1(t) & = &4\int_0^{\frac{1}{4}} z(t,x)\, dx,\qquad   
 y_2(t) =\int_\frac{1}{2}^{\frac{3}{4}} z(t,x)\, dx,
\end{eqnarray*}
%\end{subequations}
on the unit interval. The measurements $y_1(t)$ and $y_2(t)$ should asymptotically track the reference signals $y_{ref}^1(t)=\sin(2t)+1$ and $y_{ref}^2(t)=\cos(2t)$, respectively.
A suitable choice for the ring of stable transfer functions is $\stable=\Hinfty$. In what follows, the controller parameters of Design procedure \ref{des:DesignCoprime} are chosen.

\textbf{Step 1.} The Laplace transforms of the reference signals are $\hat{y}_1(s)=\frac{1}{s^2+4}+\frac{1}{s}$ and $\hat{y}_2(s)=\frac{s}{s^2+4}$. The poles of the signals locate at $s=0$, $s=2\i$, and $s=-2\i$. This is the minimal set of poles the controller should have. Since the poles $s=\pm 2\i$ appearing in both reference signals are simple, it is sufficient that the controller has only first order poles at these locations. This information is sufficient for constructing the internal model in Step 3, so the signal generator need not be given explicitly.

\textbf{Step 2.} A stabilizing controller and its left coprime factorization are found next. The eigenvalues of the plant are $\lambda_n=\pi^2+1-n^2\pi^2$, $n=0,1,\ldots$ The system has no other spectrum points and there are only two unstable eigenvalues $\lambda_0=\pi^2+1$ and $\lambda_1=1$. 
%There are several standard techniques to stabilize the system, e.g., we can stabilize the plant using an observer based finite-dimensional controller. 
The finite-dimensional controller is defined as
\begin{eqnarray*}
\dot{v} &=& M v+Hy\\
u &=& Fv
\end{eqnarray*}
where
\begin{eqnarray*}
M=\begin{bmatrix}
-200 & -150 \\
 0 & -50 
\end{bmatrix}, \quad 
H=\begin{bmatrix}
120 & 0 \\
1 & -1
\end{bmatrix}, \quad \text{and}\quad
F=\begin{bmatrix}
-50 & -75 \\
-50 &  75
\end{bmatrix}
\end{eqnarray*}
are found by using the techniques presented in Curtain and Salamon \cite{CurtainSalamon1986}. The approximated stability margin with this controller is $-1.70$. The transfer function of the stabilizing controller is 
\begin{eqnarray*}
\Cont_s(s)=F(sI-M)\inv H=\frac{1}{s^2 + 250 s + 10000}\begin{bmatrix}
-6075 s {-} 3.075\cdot 10^5 & -5925 s {-} 2.775\cdot 10^5\\
75 s {+} 7500 & -75 s {-} 22500
\end{bmatrix}.
\end{eqnarray*}
Its left coprime factorization $\Cont_s=X\inv Y$ can be found by using standard methods.\cite{OaraVarga1999} Here it is found by using the Matlab function \texttt{lncf}. Only the denominator matrix
\begin{eqnarray*}
X(s)=\frac{1}{s^2 + 8605 s + 9.911\cdot 10^5}\begin{bmatrix}
s^2 + 4324 s + 2.094\cdot 10^5 & -4108 s - 1.838\cdot 10^5\\
-4108 s - 2.063\cdot 10^5 & s^2 + 4531 s + 2.284\cdot 10^5
\end{bmatrix}
\end{eqnarray*}
is needed. Since the plant can be stabilized by a controller having a coprime factorization, it has a coprime factorization as well and the numerator of the right-coprime factorization needed in the next step is formally given by $N=P(I-C_sP)\inv X\inv$.

\textbf{Step 3.} Next a robustly regulating controller is constructed for the stable transfer matrix $N$. One such controller is given by
\begin{eqnarray*}
C_r(s)=-\varepsilon\left(\frac{N\inv(0)}{s}+\frac{N\inv(2\i)}{s-2\i}+\frac{N\inv(-2\i)}{s+2\i}\right)
\end{eqnarray*}
where $\varepsilon>0$ is to be chosen appropriately small.\cite{HamalainenPohjolainen2000,RebarberWeiss2003} The idea behind this controller is exactly the same as that of the controller in Step 3 of Example \ref{exa:Delay}. This verifies the solvability of the robust regulation problem. The designed controller is not only robust to the small perturbations in the plant, but also to small changes in the controller as long as the controller contains the internal model. Thus, it is possible to replace $P$ in $N=P(I-C_sP)\inv X\inv$ by an approximated system transfer function. This way one avoids calculation of the explicit plant transfer function. Here the plant is approximated using finite differences with ten points on $[0,1]$. Finally, $N\inv(0)$, $N\inv(2\i)$ and $N\inv(-2\i)$ are calculated and their elements are rounded to one decimal. Choosing $\varepsilon=0.5$ yields the robust controller
\begin{eqnarray*}
C_r(s)=\frac{1}{s^3 + 4 s}
\begin{bmatrix}
-0.9 s^2 + 0.4 s - 1.2 &-1.05 s^2 - 0.4 s - 1.4\\
0.9 s^2 - 0.2 s + 1.2 & -0.65 s^2 + 0.2 s - 1
\end{bmatrix}
\end{eqnarray*}
of the numerator $N$.
The robust controller
\begin{eqnarray*}
\Cont(s)=\Cont_s(s)+X\inv(s)\Cont_r(s).
\end{eqnarray*}
is obtained by substituting the above transfer matrices. This controller is causal since it is a proper rational transfer matrix. A minimal realization of $\Cont$ is found using the Matlab function \texttt{minreal}. The closed-loop system has stability margin $0.4132$. The eigenvalues closest to the imaginary axis are plotted in Fig. \ref{fig:Eigenvalues}.
\begin{figure}[ht]
\centering
\includegraphics[width=0.6\textwidth]{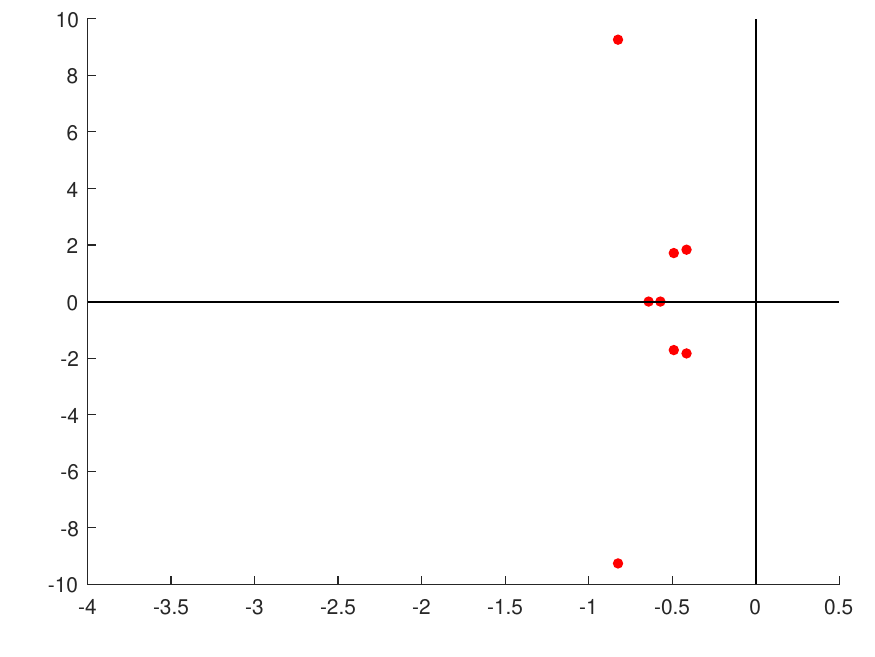}
\caption{The closed loop eigenvalues with controller $\Cont(s)$ in Example \ref{exa:HeatEq}.}
\label{fig:Eigenvalues}
\end{figure}
There are three pairs of eigenvalues in the figure two units or less away from the imaginary axis. These are the eigenvalues corresponding to the poles of $\Cont_r$. They are moving to the left from the imaginary axis when increasing $\varepsilon$ from zero to $0.5$. The remaining eigenvalues shown in the figure are the rightmost eigenvalues of the stabilized closed-loop system of Step 2 that move to the right as $\varepsilon$ is increased. Thus, the stability margin obtained with the proposed choice of $\varepsilon$ is nearly optimal with the stabilizing controller constructed in Step 2. In comparison, the stability margin obtained by using the actual closed-loop transfer matrix $\Plant$ instead of the approximated one when constructing $\Cont_r$ is $0.4588$.

Finally, the closed loop system is simulated. The approximation in the simulation is obtained by using finite differences with 150 points on $[0,1]$. Fig. \ref{fig:Regulation} shows the behavior of the measured outputs. As expected, the outputs converge asymptotically to the reference signals. The oscillation is mainly due to the eigenvalues with the largest imaginary parts shown in Fig. \ref{fig:Eigenvalues}.
\begin{figure}[ht]
\centering
\includegraphics[width=0.6\textwidth]{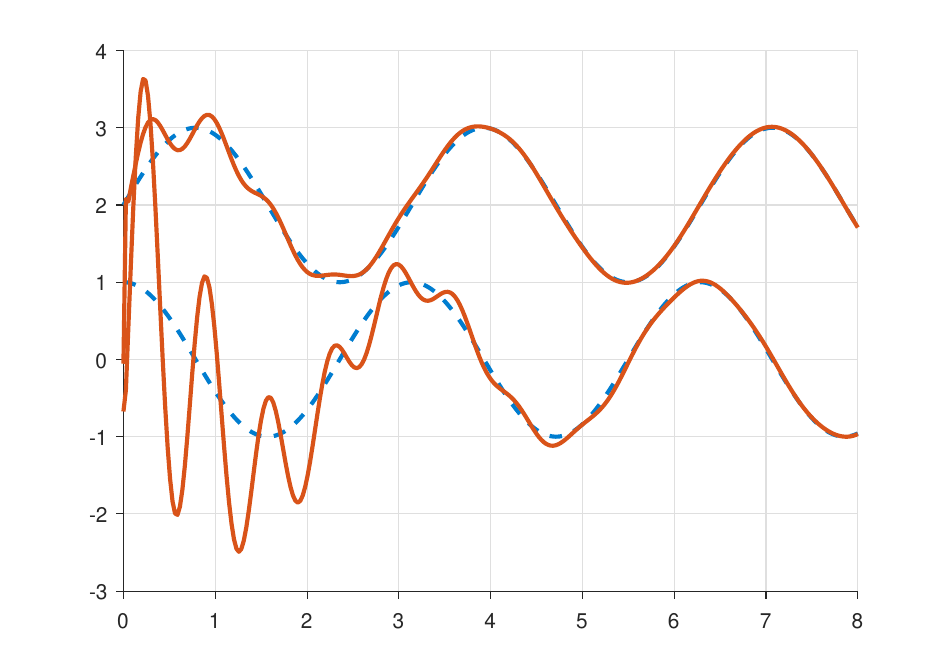}
\caption{The measured signals (solid) and the reference signals (dashed) of Example \ref{exa:HeatEq}.}
\label{fig:Regulation}
\end{figure}
\end{example}

%%%%%%%%%%%%%%%%%%%%%%%%%%%%%%%%%%%%%%%%%%%%%%%%%%%%%%%%%%%%%
%% ==> Conclusions:
\section{Concluding Remarks}\label{sec:Conclusions}

This article introduced general frequency domain theory for robust regulation using fractional representations. The main theoretical contributions were the new formulation of the internal model principle, several conditions for solvability, and the parametrization of all robustly regulating controllers. Causality considerations were included. The usefulness of the results is due to the generality assured by the minimal set of standing assumptions and not requiring the existence of coprime factorizations. Unlike the results that are related to specific rings of stable transfer functions\cite{LaakkonenPohjolainen2015, Vidyasagar}, the results presented in this article allow one to choose the stability type to work with. This is particularly important since the achievable stability type depends on the problem at hand 
%The general approach allows one to choose the type of stability to work with freely 
as was demonstrated by Example \ref{exa:Delay} in which the choice of the ring of stable transfer functions was not trivial due to the challenging unstable dynamics of the reference signals. 
% In addition, the results reveal that the existence of the internal model is indeed a fundamental feature of all robustly regulating controllers.

% guaranteeing applicability with a wide range of different stability types.
%The most general results did not use coprime factorizations and it was shown that the new results are equivalent to the ones given, e.g., in \cite{Laakkonen2016, LaakkonenPohjolainen2015, Vidyasagar} if a coprime factorization exists. %Due to generality, the results cover existing ones given in some specific ring of stable transfer functions. 
%Because of this, the results offer a unified framework to study robust regulation.

The given conditions for solvability were accompanied by design procedures for robust controllers. Although it was not possible to give details on how to accomplish the steps of the design procedures
%, e.g., stabilization or solving equations over rings, 
due to the general approach, comparing the procedures reveals some main ideas. First, some of the procedures start with a step where the internal model is simplified. This step is compulsory in Design procedure \ref{des:SimpleGen} and is particularly important in  %\ref{des:DesignStable} and \ref{des:DesignSufficient} 
the other procedures as well since without it the constructed internal model tends to be oversized, see Remark \ref{rem:SimplificationInDesign}. Secondly, Design procedure \ref{des:DesignStable} %and \ref{des:DesignSufficient} use 
gives a recursive process to construct the internal model into the controller. This technique enables one to revise an existing robustly regulating controller by adding an internal model of new unstable dynamics so that it can handle a larger class of reference signals. Thirdly, a two-step approach where one first stabilizes the plant and then constructs an internal model into the controller was used in Design procedures \ref{des:DesignGeneral}-\ref{des:DesignCoprime}.
% in different forms.
This may be particularly handy 
%in Design prodedures \ref{des:DesignSufficient} and \ref{des:DesignCoprime} 
since finding a robust controller for a stable plant can be straightforward as was seen in Example \ref{exa:HeatEq}. In Design procedure \ref{des:SimpleGen}, the order of stabilization and construction of internal model was reversed. Adding the internal model first and then stabilizing the resulting system is a straightforward method, but the downside is that one needs to stabilize the unstable dynamics of the plant and the signal generator at once.

%Finally, two examples were given to illustrate the design procedures. Solving the robust regulation probem in Example \ref{exa:Delay} was challenging due to the infinite number of poles the reference signals had on the imaginary axis. The example illustrates that the choice of the ring of stable transfer functions is not trivial, which underlines the importance of a general approach introduced in this paper. 
%Example \ref{exa:HeatEq} with a simple reference signal and boundary  controlled heat equation illustrated that the design techniques introduced for robustly regulating controllers can be applied without calculating the closed form transfer function of the plant. This is important since finding the closed form expression of an infinite-dimensional system may be hard if not impossible.

%Example \ref{exa:HeatEq} with a simple reference signal and a boundary  controlled heat equation illustrated the applicability of the proposed design methods. The steps of the procedure were carried out using standard techniques appearing in the literature. The key observation was that there is no need for calculating the closed form transfer function of the plant. This is important since finding the closed form expression of an infinite-dimensional system may be hard if not impossible.

%Example \ref{exa:HeatEq} demonstrated that applying the proposed design methods can be straighforward. The key idea was that one only required to calculate a coprime factorization for a finite-dimensional controller which is straighforward. However, no such 

The results of this article help in understanding some of the fundamental ideas in robust regulation such as the internal model principle. Whereas the results generalize
%were originally inspired by 
the existing ones that are specific to some rings of stable transfer functions, they now provide a good starting point to go back from general to specific.
% existence of the internal model is observed to be a fundamental property of any robustly regulating controller. 
%These results provide a good starting point to study 
In particular, several results require finding a generating set for a specific fractional ideal or solving matrix equations such as \eqref{eq:IMP} or \eqref{eq:Solvability}. These are not easy tasks in general and an interesting direction for future research would be to find out what one can say about their solvability in some of the most general rings of stable transfer functions such as $\Hinfty$. On the other hand, one can try to find alternative formulations of the main results in order to obtain new insights. This has been done with SISO %single-input single-output 
systems using fractional ideals.%, see Example \ref{exa:IMPSISO}.
\cite{LaakkonenQuadrat2015}
%e.g., 
%there are several time domain characterizations of the internal model principle \cite{PaunonenPohjolainen2013}. There are frequency domain frameworks, 
Two prominent frameworks for achieving further insights into robust regulation of MIMO %multi-input multi-output 
systems are the lattice approach or the geometric systems theory.\cite{Quadrat2006, Falb1999} The results concerning causality were not complete for causal plants. Therefore, stronger results on causal stabilizing controllers such as parametrization of all causal controllers would be of great interest. Parametrizations for strictly causal controllers are already available\cite{Mori2009}.

\bibliographystyle{plain}

\end{document}